\documentclass{amsart}

\usepackage{amsmath,amssymb,amsfonts,amsthm}

\usepackage{graphicx}

\usepackage{mathtools}
\usepackage{xcolor}
\usepackage{subcaption}
\usepackage{placeins}	% for \FloatBarrier 
\usepackage[linesnumbered,ruled,vlined]{algorithm2e}

\usepackage{tikz}\usetikzlibrary{patterns}
\usepackage{pgfplots}
\pgfplotsset{width=\textwidth,compat=1.9}

\newtheorem{theorem}{Theorem}[section]
\newtheorem{lemma}[theorem]{Lemma}
\newtheorem{proposition}[theorem]{Proposition}
\newtheorem{corollary}[theorem]{Corollary}

\theoremstyle{definition}
\newtheorem{definition}[theorem]{Definition}

\theoremstyle{remark}
\newtheorem{remark}[theorem]{Remark}

\numberwithin{equation}{section}

\graphicspath{{./plots/pdffiles/}}

\DeclarePairedDelimiter{\norm}{\lVert}{\rVert}	
\DeclarePairedDelimiter{\abs}{\lvert}{\rvert}	
\newcommand{\R}{\mathbb{R}}
\newcommand{\C}{\mathbb{C}}
\newcommand{\dprod}{\displaystyle\prod}

\renewcommand{\vec}[1]{\boldsymbol{#1}}		% for bold vectors
\newcommand{\barvec}[1]{\bar{\vec{#1}}}
\newcommand{\kryl}{\mathcal{P}}		% polynomial Krylov subspace
\newcommand{\rat}{\mathcal{Q}}		% rational Krylov subspace
\newcommand{\poly}{\Pi}		% space of polynomials
\newcommand{\gmf}{GMF}
\newcommand{\normal}{\mathcal{N}}	% normal distribution

\SetKwInput{KwInput}{Input}                % Set the Input
\SetKwInput{KwOutput}{Output}  

\DeclareMathOperator{\rank}{rank} 	% rank
\DeclareMathOperator{\triu}{triu} 	% upper triangular part
\DeclareMathOperator{\tril}{tril} 	% lower triangular part

\DeclareMathOperator{\diag}{diag}
\DeclareMathOperator{\vspan}{span} 	% vector span
\DeclareMathOperator{\Real}{Re} 	% real part
\newcommand\bigzero{\makebox(0,0){\text{\huge0}}}
\newcommand\bignonz{\makebox(0,0){\text{\huge*}}}

\begin{document}

\title[Generalized matrix functions with rational Krylov methods]{Computation of generalized matrix functions with rational Krylov methods}


%    author one information
\author{Angelo A. Casulli}
\address{Scuola Normale Superiore, Piazza dei Cavalieri, 7, 56126 Pisa, Italy}
\curraddr{}
\email{angelo.casulli@sns.it}
\thanks{The work of Angelo A. Casulli was partially supported by the GNCS/INdAM project “Metodi low-rank per problemi di algebra lineare con struttura data-sparse”.}

%    author two information
\author{Igor Simunec}
\address{Scuola Normale Superiore, Piazza dei Cavalieri, 7, 56126 Pisa, Italy}
\curraddr{}
\email{igor.simunec@sns.it}
\thanks{}

%    \subjclass is required.

\date{}

\dedicatory{}

%    Abstract is required.
\begin{abstract}
	We present a class of algorithms based on rational Krylov methods to compute the action of a generalized matrix function on a vector. These algorithms incorporate existing methods based on the Golub-Kahan bidiagonalization as a special case. By exploiting the quasiseparable structure of the projected matrices, we show that the basis vectors can be updated using a short recurrence, which can be seen as a generalization to the rational case of the Golub-Kahan bidiagonalization. We also prove error bounds that relate the error of these methods to uniform rational approximation.
	The effectiveness of the algorithms and the accuracy of the bounds is illustrated with numerical experiments.
\end{abstract}

\maketitle

\section{Introduction}
Generalized matrix functions (GMFs) are an extension of the notion of matrix functions based on the singular value decomposition (SVD) instead of the spectral decomposition. They were introduced for the first time in \cite{hawkins1973generalized}, with the purpose of extending the definition of matrix functions to rectangular matrices. Although the introduction of GMFs dates to the Seventies \cite{hawkins1973generalized}, they have become a more active area of research only in recent years. For instance, theoretical aspects of generalized matrix functions have been investigated in \cite{ACP16, BenziHuang19, Noferini17}, while efficient numerical methods have been developed in \cite{arrigo2016computation, aurentz2019stable}. For applications of generalized matrix functions, we direct the reader to~\cite{ACP16, arrigo2016computation, aurentz2019stable} and the references therein.

In many applications that involve standard matrix functions, it is only required to compute matrix-vector products of the form $f(A) \vec b$, where the matrix $A$ is usually large and sparse. In this case, the (expensive) computation of the whole matrix $f(A)$ can be bypassed by using methods that directly approximate the product $f(A) \vec b$, such as Krylov methods. These methods only require matrix-vector products and possibly the solution of shifted linear systems with the matrix $A$.

A similar situation arises when GMFs are involved: indeed, it is often required to compute the action of a generalized matrix function on a vector \cite{ArrigoBenzi16Edge, arrigo2016computation}, and hence it is preferable to use a method that avoids the computation of the whole GMF by means of an SVD.

This problem was recently investigated in \cite{arrigo2016computation, aurentz2019stable}, using methods based on the Golub-Kahan bidiagonalization in \cite{arrigo2016computation}, and Chebyshev polynomial interpolation in \cite{aurentz2019stable}.

In this paper, we propose a generalization of the method proposed in \cite{arrigo2016computation}, using the interpretation of the Golub-Kahan bidiagonalization in terms of Krylov subspaces.
Performing $k$ steps of the Golub-Kahan bidiagonalization of a matrix $A$ with starting vector $\vec b$ is equivalent to the simultaneous computation of orthonormal bases of the polynomial Krylov subspaces $\kryl_k(A^TA, \vec b)$ and $\kryl_k(AA^T, A \vec b)$. By replacing the polynomial Krylov subspaces with their rational counterparts, we obtain a rational Krylov method for the computation of the action of a GMF on a vector.

As can be expected by analogy with standard matrix functions, in the case of non-analytic functions and functions of low regularity these rational methods have a faster convergence than the method based on the Golub-Kahan bidiagonalization. However, their increased effectiveness comes at the cost of having to solve a linear system at each iteration, while the methods discussed in \cite{arrigo2016computation, aurentz2019stable} only require matrix-vector products.

The Golub-Kahan bidiagonalization of a matrix can be computed with a short recurrence, which relies on the fact that the projected matrix is bidiagonal. This structure is unfortunately not preserved in the rational case that we consider here. However, we are still able to construct a short recurrence to update the rational Krylov bases and the projected matrix, using the fact that the projected matrix is a quasiseparable matrix.

We also prove error bounds that link the error of the method from \cite{arrigo2016computation} based on the Golub-Kahan bidiagonalization and the rational methods introduced in this paper with, respectively, the error of uniform polynomial and rational approximation of the function $f$. These bounds are a direct generalization of the bounds for standard matrix functions, and they can be proved with the same techniques. Although the connection of GMFs to standard matrix functions is well-known, to the best of our knowledge these error bounds for the approximation of GMFs have never appeared in previous literature. 

The paper is organized as follows. In Section~\ref{sec:notation} we introduce the notation used throughout the paper.
In Section~\ref{sec:gen-mat-fun} we recall the definition of standard matrix functions and GMFs and we present some of their properties. In Section~\ref{sec-rat-kryl-methods} we briefly introduce the class of rational Krylov methods for standard matrix functions. The use of polynomial and rational Krylov methods in the context of generalized matrix functions is discussed in Section~\ref{sec:krylov-gmf}. Section~\ref{sec:error-bounds} is dedicated to the proof of the error bounds and related discussion. Some numerical experiments to compare the different methods and illustrate the error bounds are presented in Section~\ref{sec:numerical}, and Section~\ref{sec:conclusions} contains concluding remarks.

\section{Notation}
\label{sec:notation}

We denote by $\R^{m \times n}$ the space of $m \times n$ real matrices. We use bold letters for vectors, e.g.~$\vec v \in \R^n$. The entries of vector $\vec v$ are given by $v_1, \dots, v_n$, and the entries of matrix $A \in \R^{m \times n}$ are $a_{ij}$. We also use a MATLAB-like notation: $\diag(d_1, \dots, d_n)$ represents an $n \times n$ diagonal matrix with entries $d_1, \dots, d_n$ on the diagonal; for $i \le j$ and $h \le k$, we denote by $A(i:j, h:k)$ the submatrix of $A$ corresponding to row indices from $i$ to $j$ and column indices from $h$ to $k$.

We denote by $\triu(A)$ the upper triangular part of matrix $A$, and more generally by $\triu(A, k)$ the matrix with all zeroes below the $k$-th diagonal whose other entries coincide with those of $A$. Diagonals above the main diagonal are represented with a positive index, so that $\triu(A, 1)$ indicates the strictly upper triangular part of $A$. Similarly, $\tril(A)$ and $\tril(A,-1)$ denote the lower triangular and strictly lower triangular part of $A$, respectively.  We use the same notation also for rectangular matrices.
We denote by $A^+$ the Moore-Penrose pseudoinverse of a matrix $A$.

\section{Matrix functions}
\label{sec:gen-mat-fun}
The goal of this section is to define generalized matrix functions (GMFs) and introduce their main properties. We begin by recalling some basic concepts about standard matrix functions, and then we introduce \gmf s and some of their properties.

\subsection{Standard matrix functions}
The concept of matrix function is a natural way to generalize the evaluation of a function on a square matrix. For simplicity, we treat only the case of diagonalizable matrices. The general definitions and a thorough description of matrix functions can be found in the monograph~\cite{higham2008functions}.

Let $A$ be a $n\times n$ matrix. Assume that $A$ is diagonalizable, i.e.~$A=VDV^{-1}$, where $D=\diag(d_1,\dots,d_n)$. Given a function $f$ defined on the set $\{d_1,\dots,d_n\}$ the matrix function of $f$ applied on $A$ is defined as
$$ f(A)=Vf(D)V^{-1},$$
where $f(D)=\diag(f(d_1),\dots,f(d_n))$.

Equivalently, if $p(x)=\sum_{i=0}^{n-1}p_ix^i$ is a polynomial that interpolates $f$ in  $d_1,\dots,d_n$, the matrix function can be defined as

$$f(A)=p(A)=\sum_{i=0}^{n-1}p_iA^i.$$

The computation of a matrix function using the first of the two definitions requires knowledge of the eigenvalues of $A$. However, if $n$ is large, finding the eigenvalues of $A$ can be unfeasible.

A widely used technique to obtain a good approximation of $f(A)$ is to find a low-degree polynomial $q$ that approximates the interpolant polynomial $p$, and to approximate $f(A)$ with $q(A)$.
Polynomial Krylov methods, presented in Section \ref{sec-rat-kryl-methods}, use a similar strategy to approximate $f(A) \vec b$, by using $A$ and $\vec b$ to construct the polynomial $q$.

\subsection {Generalized matrix functions}
Generalized matrix function were first introduced in \cite{hawkins1973generalized}, with the purpose of extending the definition of matrix functions to rectangular matrices.
They are defined in a similar way with respect to the standard matrix functions, but the singular value decomposition is used instead of the diagonalization.

Let $A\in\R^{m\times n}$ and let $A=U\Sigma V^T$ be its SVD, where $U\in\R^{m\times m}$ and $V\in \R^{n\times n}$ are orthogonal and $\Sigma \in \R^{m\times n}$ is defined as
\begin{equation*}
	\Sigma_{i,j}=
	\begin{cases*}
		\sigma_i & if  $i=j\le r$\\
		0 &otherwise,
	\end{cases*}
\end{equation*}
where $r\le \min\{m,n\}$ is the rank of $A$ and $\sigma_1\ge\sigma_2\ge\dots\ge\sigma_r>0$ are the nonzero singular values of $A$.

Given a function $f$ defined on the set $\{\sigma_1,\dots,\sigma_r\}$ the generalized matrix function of $f$ applied on $A$ is defined as
\begin{equation*}
	f^{\diamond}(A)=Uf^{\diamond}(\Sigma) V^T,
\end{equation*}
where
\begin{equation*}
	f^\diamond(\Sigma)_{i,j}=
	\begin{cases*}
		f(\sigma_i) & if  $i=j\le r$\\
		0 &otherwise.
	\end{cases*}
\end{equation*}

Observe that a GMF can be expressed in terms of the compact SVD of the matrix $A$, that is $A=U_r\Sigma_rV_r^T$, where $U_r\in\R^{m\times r}$ and $V\in \R^{n\times r}$ have orthonormal columns, and $\Sigma_r=\diag(\sigma_1,\dots,\sigma_r)\in\R^{r\times r}$. In such case, we have
\begin{equation*}
	f^{\diamond}(A)=U_rf(\Sigma_r)V_r^T.
\end{equation*}

Since the definition of a GMF only depends on the values of $f$ on the nonzero singular values of $A$, we can always assume that $f$ is an odd function, and in particular that $f(0) = 0$.

\begin{remark}\label{remark-poly-gmf}\cite[Theorem~2.1]{aurentz2019stable}
	If $p$ is a polynomial that interpolates $f$ in $\sigma_1,\dots,\sigma_r$ it holds that $f^{\diamond}(A)=p^{\diamond}(A)$. Moreover, since $\sigma_i>0$ for $i=1,\dots,r$, we can always take $p$ as an odd polynomial, i.e.~$p(z) = q(z^2)z$ for some polynomial $q$.
\end{remark}

Next, we list some properties of GMFs that will be required in the following sections. A discussion of additional properties of generalized matrix functions can be found in~\cite{arrigo2016computation}.

\begin{lemma}\label{lemma:spd-equivalence}
	Let $S\in\R^{n\times n}$ be a symmetric matrix and let $f$ be defined on the singular values of $S$. If $S$ is positive definite, or $S$ is positive semidefinite and $f(0)=0$, it holds
	\begin{equation}
		f^\diamond(S)=f(S).
	\end{equation}
\end{lemma}
%\begin{proof}
%Since $S$ is symmetric positive semidefinite there exist $U\in\R^{n\times n}$ orthogonal and $\Sigma$ diagonal with non negative diagonal entries such that $S=U\Sigma U^T$. Assuming that the diagonal entries of $\Sigma$ are increasingly ordered (i.e., $\Sigma_{i,i}\ge \Sigma_{j,j} \text{ if } i\ge j$), such decomposition is both an SDV and a eigendecomposition, hence
%\begin{equation}
%	f^\diamond(S)=Uf^\diamond(\Sigma) U^T \quad \text{and}\quad f(S)=Uf(\Sigma) U^T .
%\end{equation}	

%From the fact that $f(0)=0$ follows that $f(\Sigma)=f^{\diamond}(\Sigma)$ and thus the thesis.
%\end{proof}

\begin{proposition}
	\label{prop:gmf-polynomial-evaluation}
	Let $p$ be an odd polynomial, i.e. we can write $p(z) = q(z^2) z$ for some polynomial $q$. Then, for any matrix $A \in \R^{m \times n}$ it holds
	\begin{equation*}
		p^\diamond(A) = q(A A^T) A = A q(A^T A).
	\end{equation*}
\end{proposition}
%\begin{proof}This proof is inspired from \cite{aurentz2019stable}.
%	Let $A=U\Sigma V^T$ be the SVD of $A$, thus $AA^T=U\Sigma^2U^T$ is an eigendecomposition of $AA^T$. From definition of matrix function we  have 
%	\begin{equation*}
%		q(A A^T) A=Uq(\Sigma^2)U^TU\Sigma V^T=Uq(\Sigma^2)\Sigma V^T=Up(\Sigma)V^T.
%	\end{equation*}
%Since $p(0)=0$, $Up(\Sigma)V^T=p^{\diamond}(A)$.

%In the same way the equivalence with $A^TA$ can be proved.
%\end{proof}

For the proof of Proposition~\ref{prop:gmf-polynomial-evaluation} we refer to \cite[Section~2.2]{aurentz2019stable}.

\begin{corollary}\label{cor:gmf-rational-evaluation}
	Let $r$ be a rational function with an odd numerator and an even denominator, i.e.,  $r(z) = \dfrac{p(z^2)}{q(z^2)} z$ for some polynomials $p$ and $q$, such that the singular values of $A \in \R^{m \times n}$ and $0$ are not roots of $q$.  Then it holds

	\begin{equation*}
		r^\diamond(A) = q(A A^T)^{-1} p(A A^T) A = A q(A^T A)^{-1} p(A^T A).
	\end{equation*}
\end{corollary}
\begin{proof}
	The proof is practically the same as the proof of Proposition~\ref{prop:gmf-polynomial-evaluation}. We report it for completeness. 
	Let $A = U \Sigma V^T$ be an SVD of $A$. Then we have $AA^T = U \Sigma^2 U^T$, and more generally $s(AA^T) = U s(\Sigma^2) U^T$ for any polynomial $s$.

	Since $0$ is not a root of $q$, we have $r(0) = 0$ and hence by Lemma~\ref{lemma:spd-equivalence} we have
	\begin{equation*}
		r^\diamond(A) = U r^\diamond(\Sigma) V^T = U r(\Sigma) V^T = U q(\Sigma^2)^{-1}p(\Sigma^2) \Sigma V^T.
	\end{equation*}
	Note that $q(\Sigma^2)^{-1}$ is well defined because the singular values of $A$ are not roots of $q$.
	Now, multiplying by $U^TU = I$ we obtain
	\begin{equation*}
		r^\diamond(A) = U q(\Sigma^2)^{-1} U ^T U p(\Sigma^2) U^T U \Sigma V^T = q(AA^T)^{-1} p(AA^T) A.
	\end{equation*}
	The second identity can be obtained in a similar way.
\end{proof}

The following proposition gives a formulation of the above results for general functions. See also~\cite[Theorem~10]{arrigo2016computation} for an alternative formulation of the same identity.

\begin{proposition}
	\label{prop:gmf-general-equivalence}
	Let $A\in\R^{m\times n}$ and let $f$ be a function defined on the nonzero singular values of $A$. Defining $g(z)=\frac{f(\sqrt{z})}{\sqrt{z}}$, for $z \ne 0$, it holds
	\begin{equation*}
		f^\diamond(A) = g^{\diamond}(A A^T) A = A g^{\diamond}(A^T A).
	\end{equation*}
	Moreover, if $\displaystyle\lim_{z \to 0} \tfrac{f(z)}{z} = 0$, we can define $g(0) = 0$. Then $g(AA^T)$ and $g(A^TA)$ are well defined, and we have
	\begin{equation*}
		f^\diamond(A) = g(A A^T) A = A g(A^T A).
	\end{equation*}
\end{proposition}
\begin{proof}
	Let $p(z)=q(z^2) z$ be an odd polynomial that interpolates $f$ in the nonzero singular values of $A$. From Proposition \ref{prop:gmf-polynomial-evaluation} and Remark \ref{remark-poly-gmf} we have
	\begin{equation*}
		f^{\diamond}(A)=p^{\diamond}(A)=q(A A^T) A = A q(A^T A).
	\end{equation*}
	Since $p$ interpolates $f$ in the nonzero singular values of $A$, $q$ interpolates $g$ in the squares of the nonzero singular values of $A$, which are the nonzero singular values of $AA^T$ and $A^TA$. Hence $q^\diamond(AA^T)=g^{\diamond}(AA^T)$ and $q^\diamond(A^TA)=g^{\diamond}(A^TA)$.

	If $g(0) = 0$, by Lemma~\ref{lemma:spd-equivalence} we also have $g^\diamond(A^TA)=g(A^TA)$ and $g^\diamond(AA^T)=g(AA^T)$.
\end{proof}

\begin{remark}
	If $AA^T$ is positive definite, by Lemma~\ref{lemma:spd-equivalence} it holds $f^\diamond(A) = g(AA^T) A$ without the assumption $\displaystyle\lim_{z \to 0} \tfrac{f(z)}{z}$, and similarly if $A^TA$ is positive definite we have $f^\diamond(A) = Ag(A^TA)$. Note that we could also artificially define $g(0) = 0$ without any assumptions on $f$, since the definition of a GMF only depends on the nonzero singular values of the matrix. However, this would cause $g$ to be discontinuous at $0$ and it would not be very useful in practice.
\end{remark}

The following proposition links $f^\diamond(A)$ with $f^\diamond(A^T)$, which will be useful when $A$ is rectangular. We recall that $A^+$ denotes the Moore-Penrose pseudoinverse of $A$, which can be defined as $A^+ = h^\diamond(A^T) = h^\diamond(A)^T$ for $h(z) = z^{-1}$. The more general statement in Proposition~\ref{prop:link-gmf-a-and-gmf-at} can be seen as a generalization of~\cite[Proposition~7(iv)]{arrigo2016computation} and~\cite[Theorem~4(d)]{hawkins1973generalized}.
\begin{proposition}
	\label{prop:link-gmf-a-and-gmf-at}
	Let $A \in \R^{m \times n}$ and let $f$ be a function defined on the nonzero singular values of $A$. Then it holds
	\begin{equation*}
		f^\diamond(A) = (A^+)^T f^\diamond(A^T)A.
	\end{equation*}
	More generally, assume that $f(z) = g(z) h(z) k(z)$, where $g, h, k$ are functions defined on the nonzero singular values of $A$. Then it holds
	\begin{equation*}
		f^\diamond(A) = g^\diamond(A) h^\diamond(A^T) k^\diamond(A).
	\end{equation*}
\end{proposition}
\begin{proof}
	We directly prove the generalized version, since the first statement simply follows by taking $g(z) = z^{-1}$, $h(z) = f(z)$ and $k(z) = z$.

	Let $A$ have the singular value decomposition $A = U \Sigma V^T$, where $U \in \R^{m \times m}$ and $V \in \R^{n \times n}$ are orthogonal, and $\Sigma \in \R^{m \times n}$ has the singular values of $A$ on the main diagonal. We have:
	\begin{align*}
		g^\diamond(A) h^\diamond(A^T) k^\diamond(A) & = U g^\diamond(\Sigma) V^TV h^\diamond(\Sigma^T) U^TU k^\diamond(\Sigma) V^T \\
		& = U g^\diamond(\Sigma) h^\diamond(\Sigma)^T k^\diamond(\Sigma) V^T           \\
		& = U f^\diamond(\Sigma) V^T = f^\diamond(A),
	\end{align*}
	where we used that $g^\diamond(\Sigma) h^\diamond(\Sigma)^T k^\diamond(\Sigma) = f^\diamond(\Sigma)$, which can be verified directly.
\end{proof}

\begin{remark}
	The same proof of Proposition~\ref{prop:link-gmf-a-and-gmf-at} can be used to show that, if $f(z) g(z) = h(z) k(z)$, then it holds
	\begin{equation*}
		f^\diamond(A^T) g^\diamond(A) = h^\diamond(A^T) k^\diamond(A).
	\end{equation*}
	In particular we have $A^T f^\diamond(A) = f^\diamond(A^T) A$.
\end{remark}

\section{Rational Krylov methods} \label{sec-rat-kryl-methods}

The class of Krylov methods provides an efficient way to compute approximations to expressions of the form $f(A) \vec b$. The main idea behind these methods is to construct a low dimensional subspace $\mathcal{S}_k \subset \R^n$ for some integer $k \ll n$ using information from $A$ and $\vec b$, and then to approximate $f(A)\vec b$ with an appropriate vector from $\mathcal{S}_k$.

A popular choice for the approximation subspace $\mathcal{S}_k$ is the \emph{polynomial Krylov subspace}
\begin{equation*}
	\kryl_k(A,\vec b) = \vspan\{ \vec b, A \vec b, \dots, A^{k-1} \vec b\} = \{ p_{k-1}(A) \vec b : p \in \poly_{k-1}\},
\end{equation*}
where $\poly_{k-1}$ denotes the set of polynomials of degree $\le k-1$.

More generally, using a sequence of \emph{poles} $\{ \xi_k \}_{k \ge 1} \subseteq (\C\cup \{\infty\}) \setminus {\sigma(A)}$, one can define the \emph{rational Krylov subspace}
\begin{equation*}
	\rat_k(A,\vec b) = q_{k-1}(A)^{-1} \kryl_k(A, \vec b) = \Big\{ r_{k-1}(A) \vec b : r_{k-1}(z) = \frac{p_{k-1}(z)}{q_{k-1}(z)}, \text{with } p_{k-1} \in \poly_{k-1}\Big\},
\end{equation*}
where $q_{k-1}(z) = \dprod_{j = 1}^{k-1}(1 - z/\xi_j)$. In the case when all poles are located at $\infty$, we have $q_{k-1}(z) \equiv 1$ and hence we recover the polynomial Krylov subspace $\kryl_k(A, \vec b)$. It is easy to verify that the Krylov subspaces $\rat_k(A, \vec b)$ form a nested sequence, and that $\dim \rat_k(A, \vec b) = k$ as long as $k$ is smaller than the \emph{invariance index} $K$ of the sequence, i.e. the smallest integer such that $\rat_K(A, \vec b) = \rat_{K+1}(A, \vec b)$ (or, equivalently, $\kryl_K(A, \vec b) = \kryl_{K+1}(A, \vec b)$).

For $k \le K$, an orthonormal basis $\{ \vec v_1, \dots, \vec v_k \}$ of $\rat_k(A, \vec b)$ can be computed with the rational Arnoldi algorithm, introduced by Ruhe in \cite{Ruhe94}.
In the basic algorithm, the first basis vector is chosen as $\vec v_1 = \vec b/\norm{\vec b}_2$.
Then, given a set of vectors $\{ \vec v_1, \dots, \vec v_j \}$ which form an orthonormal basis of $\rat_j(A, \vec b)$, the next basis vector $\vec v_{j+1}$ is computed by orthonormalizing the vector $(I - \frac1{\xi_{j+1}}A)^{-1} A \vec v_j$ against the previously computed basis vectors.
To prevent the algorithm from failing, it is required that $(I - \frac1{\xi_{j+1}}A)^{-1} A \vec v_j \in \rat_{k+1}(A, \vec b) \setminus \rat_k(A, \vec b)$; this property is almost always satisfied in practice, however there are no theoretical guarantees that it holds.
An approach for finding a vector $\vec w_j$ that guarantees that $(I - \frac1{\xi_{j+1}}A)^{-1} \vec w_j \in \rat_{k+1}(A, \vec b) \setminus \rat_k(A, \vec b)$ and is (near-)optimal in some sense was recently discussed in \cite{BerljafaGuettel17}, using the notion of \emph{continuation pairs} $(\eta_j/\rho_j, \vec t_j)$: in general such a vector $\vec w_j$ is of the form $(\rho_m A - \eta_m I) V_j \vec t_j$, where $V_j = [\vec v_1 \dots \vec v_j] \in \C^{n \times j}$.

Using the matrix with orthonormal columns $V_k = [\vec v_1  \dots \vec v_k ]\in \C^{n \times k}$, we can compute the following approximation to $f(A) \vec b$ from the subspace $\rat_k(A, \vec b)$:
\begin{equation}
	\label{eqn:rational-krylov-approximation}
	\barvec y_k = V_k f(V_k^*AV_k) V_k^* \vec b = V_k f(A_k) \vec e_1,
\end{equation}
where $A_k = V_k^* A V_k$ is the projection of $A$ on the subspace $\rat_k(A, \vec b)$, and $\vec e_1$ denotes the first vector of the canonical basis of $\C^n$. The accuracy of the approximation $\barvec y_k$ is largely dependent on the pole sequence $\{\xi_k\}_{k \ge 1}$. The problem of choosing a sequence of poles that is effective for a particular function $f$ and a given set containing the spectrum of $A$ has often been discussed in the literature: we refer, for instance, to \cite{Guettel13} and the references therein.

Some specific sequences of poles lead to special cases of the rational Krylov subspaces $\rat_k(A, \vec b)$: if all the poles are equal to $\xi \in \C \setminus \sigma(A)$, then $\rat_(A, \vec b)$ is a \emph{Shift-and-Invert Krylov subspace},
\begin{equation*}
	Q_k(A, \vec b) = \kryl_k\left(\left(I - \frac 1 \xi A\right)^{-1}, \vec b\right),
\end{equation*}
that was first investigated for the computation of matrix functions in \cite{MoretNovati04, VDEHochbruck06}. If the poles alternate between $0$ and $\infty$, we obtain the \emph{extended Krylov subspace}, introduced in \cite{DruskinKnizhnerman98}, which is of the form
\begin{equation*}
	\rat_{2k}(A, \vec b) = A^{-k}\kryl_{2k}(A, \vec b) = \kryl_{2k}(A, A^{-k}\vec b).
\end{equation*}
We refer to \cite{GuettelThesis} for an extensive discussion on rational Krylov methods for the computation of matrix functions.

\section{Krylov methods for GMFs}
\label{sec:krylov-gmf}
The computation of $f^{\diamond}(A) \vec b$ by using an SVD of $A$ can be unfeasible if the size of $A$ is large. A possible way to approximate the product of a generalized matrix function times a vector is to use two rectangular matrices with orthonormal columns to project the matrix $A$ onto a smaller space and then to compute the generalized matrix function of the projected matrix: let $k\ll n$ and let $U_k,V_k\in\R^{n\times k}$ with orthonormal columns,  $B_k\in\R^{k\times k}$ such that
$A\approx U_kB_kV_k^T$, then
\begin{equation}
\label{eqn:gmf-approx-polynomial-0}
	f^{\diamond}(A) \vec b\approx U_k f^\diamond(B_k)V_k^T \vec b,
\end{equation}
where the matrix $f^\diamond (B_k)$ can be computed by means of the singular value decomposition.

In this section we describe how to compute such a projected matrix using the Golub-Kahan bidiagonalization, which is equivalent to using a polynomial Krylov method on the matrices $A^TA$ and $AA^T$. This strategy corresponds to the ``third approach'' discussed in \cite[Section~5.4]{arrigo2016computation}; the numerical results of \cite{arrigo2016computation} indicate that \eqref{eqn:gmf-approx-polynomial-0} is often more effective than the other approaches they propose, which are based on Gauss and Gauss-Radau quadrature formulas.

In Sections~\ref{subsec:rat-kry-gmf} and~\ref{subsec:short-term-rec} we generalize this approach to the rational Krylov case and we show that a short recurrence like the one of the Golub-Kahan bidiagonalization can be obtained in the rational case too.

\subsection{Golub-Kahan bidiagonalization}
The first method we describe for the computation of a truncated SVD is the Golub-Kahan bidiagonalization introduced for the first time in 1965.

\begin{theorem} \label{thm:householder-bidiagonalization}
	Let $A\in \R^{m\times n}$, with $m>n$. There exist orthogonal matrices $P\in \R^{m\times m}, Q\in \R^{n\times n}$ such that
	\begin{equation}
		P^TAQ=B=
		\begin{bmatrix}
			\alpha_1 & \beta_1  & \dots   & \dots        & 0           \\
			0        & \alpha_2 & \beta_2 & \dots        & \vdots      \\
			\vdots   & \ddots   & \ddots  & \ddots       & \vdots      \\
			\vdots   &          & 0       & \alpha_{n-1} & \beta_{n-1} \\
			0        & \dots    & \dots   & 0            & \alpha_n    \\
			0        & \dots    & \dots   & \dots        & 0           \\
			\vdots   &          &         &              & \vdots      \\
			0        & \dots    & \dots   & \dots        & 0
		\end{bmatrix}	.
	\end{equation}
	Moreover the first column of $Q$ can be chosen arbitrarily.
\end{theorem}
The proof of Theorem~\ref{thm:householder-bidiagonalization} is constructive and it is usually called Householder bidiagonalization process. It can be found in \cite[Section 5.4]{golub2013matrix}.

In the case of large matrices the full bidiagonalization is too expensive. The goal of the Golub-Kahan bidiagonalization is to extract good approximations of singular values and singular vectors before the full bidiagonalization is completed.

Let $$P=[\vec p_1|\dots|\vec p_m],\qquad Q=[\vec q_1|\dots| \vec q_n]$$ be a column partitioning of $P$ and $Q$. From Theorem~\ref{thm:householder-bidiagonalization} it follows that $AQ=PB$ and $A^TP=QB^T$; using these relations we have
\begin{equation} \label{eqn:Golub-kahan-recurrence}
	\begin{aligned}
		A\vec q_k   & =\alpha_k\vec p_k+\beta_{k-1}\vec p_{k-1}, \\
		A^T\vec p_k & =\alpha_k\vec q_k+\beta_k\vec q_{k+1},
	\end{aligned}
\end{equation}
for $k=1,\dots,n,$ with the convention that $\beta_0\vec p_0=0$ and $\beta_n\vec q_{n+1}=0$.

Let $r_k=A\vec q_k-\beta_{k-1}\vec p_{k-1}.$ Using the orthogonality of the columns of $P$, we have
\begin{equation}\label{eq_alpha}
	\alpha_k=\pm\norm{\vec r_k}_2, \qquad \vec p_k=\frac{\vec r_k}{\alpha_k} \quad \text{(if }\alpha_k \neq 0).
\end{equation}
Similarly defining $\vec s_k=A^T\vec p_k-\alpha_k\vec q_k$ we have
\begin{equation}\label{eq_beta}
	\beta_k=\pm\norm{\vec s_k}_2, \qquad \vec q_{k+1}=\frac{\vec s_k}{\beta_k} \quad \text{(if }\beta_k \neq 0).
\end{equation}
Hence, given $\vec p_{k-1},\vec q_k,\beta_{k-1}$, we can compute $\vec p_{k}$, $\vec q_{k+1}$, $\beta_{k}$ and $\alpha_k$.

Defining $P_k=[\vec p_1|\dots|\vec p_k],$ $Q_k=[\vec q_1|\dots| \vec q_k]$ and
\begin{equation*}
	B_k=
	\begin{bmatrix}
		\alpha_1 & \beta_1  & 0       & \dots        & 0           \\
		0        & \alpha_2 & \beta_2 & \ddots       & \vdots      \\
		\vdots   & \ddots   & \ddots  & \ddots       & 0           \\
		\vdots   &          & 0       & \alpha_{k-1} & \beta_{k-1} \\
		0        & \dots    & \dots   & 0            & \alpha_k    \\
	\end{bmatrix},
\end{equation*}
after the $k$-th step of \eqref{eq_alpha} and \eqref{eq_beta} we have
\begin{equation}
	\begin{aligned}
		AQ_k   & =P_kB_k,                  \\
		A^TP_k & =Q_kB_k^T+\vec s_k e_k^T,
	\end{aligned}
\end{equation}
assuming $\alpha_k>0$. It can be shown that
\begin{equation}
	\begin{aligned}
		\text{span}\{\vec q_1,\dots, \vec q_k\} & =\kryl_k(A^TA,\vec q_1),  \\
		\text{span}\{\vec p_1,\dots, \vec p_k\} & =\kryl_k(AA^T,A\vec q_1), \\
	\end{aligned}
\end{equation}
thus the convergence of the Golub-Kahan bidiagonalization follows from the convergence of the Lanczos method applied on $A^TA$ and $AA^T$.

For a given $k$, we can approximate $A$ with $P_kB_kQ_k^T$, and hence we can compute an approximation of the SVD of $A$ by computing the SVD of $B_k$.
For further information on the Golub-Kahan bidiagonalization we refer to \cite[Chapter 10]{golub2013matrix}.

The vector $\vec y = f^\diamond(A) \vec b$ can be approximated with the expression
\begin{equation}
	\label{eqn:gmf-approx-polynomial}
	\barvec y_k = P_k f^\diamond(B_k) Q_k^T \vec b = \norm{\vec b}_2 P_k f^\diamond(B_k) \vec e_1 \: \in \: \kryl_k(AA^T, A \vec b).
\end{equation}

We refer to this approximation as a polyomial Krylov method for GMFs.

\subsection {Rational Krylov methods for GMFs} \label{subsec:rat-kry-gmf}
As we saw in the previous section, the Golub-Kahan bidiagonalization computes orthonormal bases for the polynomial Krylov subspaces $\kryl_k(A^TA, \vec b)$ and $\kryl_k(AA^T, A \vec b)$. By analogy with that approach, in this section we propose to compute an approximation to $f^\diamond(A) \vec b$ using the rational Krylov subspaces
\begin{equation}
	\rat_k(A^T A,\vec b) \qquad \text{and} \qquad \rat_k(A A^T, A \vec b),
\end{equation}
where $q_{k-1}(z) = \dprod_{j=1}^{k-1}(1 - z/\xi_j)$, for a given pole sequence $\{ \xi_j \}_{j \ge 1}$.

Assume that we have constructed two matrices with orthonormal columns $P_k$ and $Q_k$, such that $\vspan(P_k) = \rat_k(AA^T, A\vec b)$ and $\vspan(Q_k) = \rat_k(A^TA, \vec b)$. Then, defining $B_k = P_k^T A Q_k$, by analogy with the polynomial Krylov approach we can introduce the vector
\begin{equation}\label{eqn:gmf-approx-rational}
	\barvec y_k = P_k f^\diamond(B_k) Q_k^T \vec b=P_k f^\diamond(B_k) \vec e_1,
\end{equation}
which is an approximation to $f^\diamond(A) \vec b$ from the subspace $\rat_k(A A^T, A \vec b)$.

First of all, notice that it is sufficient to compute only one of the two rational Krylov subspaces: indeed, since it holds $A \rat_k(A^TA, \vec b) = \rat_k(A A^T, A \vec b)$, we can compute an orthonormal basis of the subspace $\rat_k(AA^T, A \vec b)$ simply by orthonormalizing the columns of $A Q_k$. This is equivalent to computing a QR decomposition $A Q_k = W_k R_k$, where $W_k$ has orthonormal columns and $R_k$ is upper triangular, so we can set $P_k = W_k$. Moreover, notice that we also have
\begin{equation*}
	B_k = P_k^T A Q_k = W_k^TW_k R_k = R_k,
\end{equation*}
i.e.~with the QR decomposition we also recover the matrix $B_k$, without the need to project $A$ explicitly. The basis $Q_k$ of the subspace $\rat_k(A^T A, \vec b)$ can be computed by applying the rational Arnoldi algorithm to the matrix $A^T A$ with initial vector~$\vec b$.

\subsection{Short recurrence for rational Golub-Kahan algorithm}\label{section:Rat-Golub_Kahan}
\label{subsec:short-term-rec}
In the Golub-Kahan bidiagonalizazion (without reorthogonalization), we can compute the last column of $P_k$, $Q_k$ and $B_k$ just by knowing a few previous columns of $P_k$ and $Q_k$, by means of the equations \eqref{eqn:Golub-kahan-recurrence}. This short recurrence is possible because of the bidiagonal structure of the matrix $B_k$, that unfortunately is not preserved when we perform a rational Krylov method.

In this section we show that, if a rational Krylov method is used, the matrix~$B_k$ is a quasiseparable matrix (see Definition~\ref{def:quasiseparable-matrix}). This structure extends the bidiagonal form that is obtained during the Golub-Kahan bidiagonalization. Using such structure we build a short recurrence that allows us to update the matrix $P_k$ and the matrix $B_k$ avoiding the full orthogonalization related to the computation of the QR factorization of the matrix $AQ_k.$

\begin{definition}\label{def:semiseparable-matrix}
	A matrix $S\in \R^{n\times n}$ is called a semiseparable matrix if all submatrices
	taken out of the lower and upper triangular part of the matrix have rank at most 1, that is
	\begin{equation*}
		\rank S(i : n, 1 : i) \le 1 \quad \text{and} \quad \rank S(1 : i, i : n) \le 1,
	\end{equation*}
	for every $i=1,\dots,n$.
	Moreover, S is called a generator representable semiseparable
	matrix if the lower and upper triangular parts of the matrix are derived from a rank~1 matrix, that is
	\begin{equation*}
		\tril(S) = \tril(\vec u\vec v^T )\quad \text{and} \quad \triu(S) = \triu(\vec p\vec q^T ),
	\end{equation*}
	for $\vec u,\vec v,	\vec p,\vec q\in \R^n$.
\end{definition}

\begin{definition}\label{def:quasiseparable-matrix}
	A matrix S is called a quasiseparable matrix if all the subblocks
	taken out of the strictly lower triangular part of the matrix (respectively, the strictly upper triangular part) are of rank at most 1, that is
	\begin{equation*}
		\rank S(i + 1 : n, 1 : i) \le 1 \quad \text{and} \quad \rank S(1 : i, i + 1 : n) \le 1,
	\end{equation*}
	for every $i=1,\dots,n$.

\end{definition}

The following theorem from \cite[Section~1.5.2]{vandebril2007matrix} gives us a complete characterization of invertible semiseparable matrices.

\begin{theorem}\label{thm:inverse-semiseparable}
	The inverse of an invertible tridiagonal matrix is a semiseparable
	matrix, and vice versa. Moreover, the inverse of an invertible irreducible tridiagonal matrix is a generator representable semiseparable matrix and vice versa.
\end{theorem}

As it has been proved in \cite[Section~5.2]{GuettelThesis}, if we perform a rational Krylov algorithm on a symmetric matrix we obtain a particular equivalence called rational Arnoldi decomposition. See also \cite[Section~2]{berljafa2015generalized} and \cite[eq.~(2.2)]{PPS21}.

\begin{theorem}\label{thm:symmetric-RAD}
	Given a symmetric matrix $A\in\R^{n\times n}$ and a vector $\vec b\in \R^n$, let $\rat_{k+1}(A,\vec b)$ be the rational Krylov space with nonzero poles $\{\xi_1,\dots,\xi_k\}$  and assume that $k$ is less than the invariance index of the Krylov subspace. Let $Q_{k+1}\in\R^{n\times (k+1)}$ be the matrix with orthonormal columns generated by the Arnoldi algorithm, such that $\vspan (Q_{k+1}) =\rat_{k+1}(A,\vec b)$. Then the following relation holds:
	\begin{equation} \label{eqn:Rational-Arnoldi-decomposition}
		AQ_{k+1}\underline{K_k}=Q_{k+1}\underline{H_k}, \quad \text{with} \quad \underline{K}_k=\underline I_k+D_k\underline H_k,
	\end{equation}
	where $\underline{H_k}\in \R^{(k+1)\times k}$ is a full rank tridiagonal irreducible symmetric matrix, $\underline{I}_k$ is the $(k+1)\times k$ identity matrix and $D_k=\text{diag}(0,\frac{1}{\xi_1},\dots,\frac{1}{\xi_{k-1}})$, where $\frac{1}{\infty}=0.$
\end{theorem}

Starting form \eqref{eqn:Rational-Arnoldi-decomposition}, we are going to prove that the projection of the symmetric matrix $A$ on the Krylov subspace (i.e., $Q^T_{k+1}AQ_{k+1}$) is the sum of a diagonal matrix and a semiseparable matrix.

\begin{theorem}{\label{thm:semiseparable-projection}}
	Let $A\in \R^{n\times n}$ be a symmetric matrix and let $Q_{k+1}\in\R^{n\times (k+1)}$ be the matrix with orthonormal columns generated by the rational Arnoldi algorithm using poles $\{\xi_1,\dots,\xi_k\}$ different from zero and infinity. Assuming that $k+1$ is less than the invariance index, we have
	$$J_{k+1}:=Q_{k+1}^TAQ_{k+1}=S+\tilde D_{k},$$
	where $\tilde D_{k}=\diag(0,\xi_1,\dots,\xi_{k})$ and $S$ is a symmetric generator representable semiseparable matrix.
\end{theorem}
\begin{proof}
	Using $\xi_{k+1}=\infty$, from Theorem~\ref{thm:symmetric-RAD} we obtain the relation
	\begin{equation}\label{1}
		AQ_{k+2}\underline K_{k+1}=Q_{k+2}\underline H_{k+1},
	\end{equation}
	where $\underline K_{k+1}$ and $\underline H_{k+1}$ are tridiagonal, $\vec e_{k+2}^T\underline K_{k+1}=\vec 0^T$ and $	\vec e_{k+2}^T\underline H_{k+1}$ is a multiple of $\vec e_{k+1}^T$.
	Hence, multiplying \eqref{1} on the left by $Q_{k+2}^T$ and taking the first $k+1$ columns and rows, we have  $$J_{k+1}=H_{k+1}K_{k+1}^{-1}.$$

	%Proving that $J_{k+1}=H_{k+1}K_{k+1}^{-1}=\hat S+\tilde D_k$ with  $\hat S$ symmetric generator representable semiseparable matrix, we conclude since $J_k$ is a principal submatrix of $J_{k+1}.$ 

	Let us define $\hat{D}_k=\diag(\gamma, \xi_1,\dots, \xi_k)$, for $\gamma\in \R$.  Notice that, since the first column of $Q_{k+1}$ is not an eigenvector of $A$, the entry in position (1,2) of $J_{k+1}$ has to be different from 0. From this it can be noticed that $J_{k+1}-\tilde D_k$ is a symmetric generator representable semiseparable matrix if and  only if $J_{k+1}-\hat D_k$ is. Moreover, taking $\gamma \neq h_{1,1}$ and $\gamma \neq 0$  we have that $J_{k+1}-\hat D_k$ is invertible and its inverse can be computed as follows:
	\begin{equation*}
		\begin{aligned}
			\left(H_{k+1}K_{k+1}^{-1}-\hat D_k\right)^{-1} & =\left(-\hat D_k(K_{k+1}-\hat D_k^{-1}H_{k+1})K_{k+1}^{-1}\right)^{-1}= \\
			                                               & =-K_{k+1}(K_{k+1}-\hat D_k^{-1}H_{k+1})^{-1}\hat D_k^{-1}.
		\end{aligned}
	\end{equation*}
	Since $K_{k+1}={I}_{k+1}+D_k H_{k+1}$  where $D_k=\diag(0,\frac{1}{\xi_1},\dots,\frac{1}{\xi_k}),$ we have
	\begin{equation*}
		\begin{aligned}
			\left(H_{k+1}K_{k+1}^{-1}-\hat D_k\right)^{-1} & =-K_{k+1}({I}_{k+1}+(D_k-\hat D_k^{-1})H_{k+1})^{-1}\hat D_k^{-1}                     \\
			                                               & =-K_{k+1}({I}_{k+1}-\frac1 \gamma \vec e_1 \vec e_1^TH_{k+1})^{-1}\hat D_k^{-1}       \\
			                                               & =-K_{k+1}({I}_{k+1}+\frac{1}{\gamma-h_{1,1}}\vec e_1 \vec e_1^TH_{k+1})\hat D_k^{-1}= \\
			                                               & =-(K_{k+1}+\frac{1}{\gamma-h_{1,1}}K_{k+1}(\vec e_1 \vec e_1^T)H_{k+1})\hat D_k^{-1}.
		\end{aligned}
	\end{equation*}
	The third equality follows from the Sherman-Morrison formula, using the fact that $\gamma \ne h_{1,1}$. This also shows that the matrix $J_{k+1} - \hat D_k$ is indeed invertible.

	The obtained matrix is an irreducible tridiagonal matrix since $K_{k+1}$ and $H_{k+1}$ have such structure and $\gamma \ne 0$. Hence, using Theorem~\ref{thm:inverse-semiseparable}, we have that $J_{k+1}-\hat D_k$ is a generator representable semiseparable matrix, and therefore also $J_{k+1} - \tilde D_k$ is.
\end{proof}

The following corollary generalizes the statement of Theorem~\ref{thm:semiseparable-projection} to the case with poles at $\infty$.

\begin{corollary}
	If there exists a pole $\xi_i=\infty$, the projected matrix $J_{k+1}=Q_{k+1}^TAQ_{k+1}$ is still a quasiseparable matrix.
\end{corollary}
\begin{proof}
	Let $\{\xi^{(j)}\}_{j\in \mathbb{N}}$ be a sequence of real numbers outside of the convex hull of $\sigma(A)$ such that
	$$\lim_{j\rightarrow \infty}\xi^{(j)}=\infty,$$
	and let $J_{k+1}^{(j)}$ be the projected matrix obtained by using poles equal to $$\{\xi_1,\dots,\xi_{i-1},\xi^{(j)},\xi_{i+1},\dots,\xi_k\}.$$
	Since the basis computed by the rational Arnoldi process depends continuously on the poles, we have that
	$$\lim_{j\rightarrow \infty}J_{k+1}^{(j)}=J_{k+1}.$$

	From Theorem~\ref{thm:semiseparable-projection} we have that the matrices $J_{k+1}^{(j)}$ are quasiseparable for all $j$. Since the quasiseparable matrices are a closed set \cite[Section~1.4.1]{vandebril2007matrix}, we have the thesis.
\end{proof}

Notice that, if the matrix $A$ is symmetric positive semidefinite and $k+1$ is less than the invariance index, the projected matrix $J_k=Q_k^TAQ_k$ has to be positive definite. Indeed, if there exists a vector $\vec x\neq \vec 0$ such that $J_k\vec x = \vec 0$, we have that $AQ_k \vec x = \vec 0$. In particular, since $Q_k \vec x\in q_{k-1}(A)^{-1}\kryl_k(A,b)$, where $q_{k-1}$ is as defined in Section~\ref{sec-rat-kryl-methods}, there exist $\alpha_0,\dots,\alpha_{j}$, $j \le k-1$, with $\alpha_j \ne 0$ such that
\begin{equation*}
	Q_k \vec x=q_{k-1}(A)^{-1} \sum_{i=0}^{j}\alpha_i A^i \vec b,
\end{equation*}
and so
\begin{equation*}
	\vec 0 = AQ_k \vec x = q_{k-1}(A)^{-1}\sum_{i=0}^{j}\alpha_i A^{i+1}\vec b.
\end{equation*}
This implies that $A^{j+1} \vec b \in \kryl_j(A, \vec b)$, but this is impossible because $k+1$ is less than the invariance index. In particular, this implies the existence of the Cholesky factorization of the matrix $J_k$ if the matrix $A$ is symmetric positive semidefinite.

The matrix $B_k = P_k^T A Q_k$ we are interested in when using rational Krylov methods for GMFs is exactly the transpose of the Cholesky factor of the matrix $J_k = Q_k^T (A^TA) Q_k$. Indeed $B_k$ is upper triangular and
$$J_k=(AQ_k)^TAQ_k=(P_kB_k)^T(P_kB_k)=B_k^TB_k.$$
The following proposition gives us that the matrix $B_k$ is the upper triangular part of a rank one plus diagonal matrix.
\begin{proposition}Let $A\in\R^{n\times n}$ be symmetric positive definite and let $L\in \R^{n\times n}$ be its Cholesky factor (i.e., $L$ is lower triangular and $A=LL^T$).
	If there exist $\vec u ,\vec v\in \R^{n}$ such that $\tril(A,-1)=\tril(\vec v\vec u^T,-1)$, then $\tril(L,-1)=\tril(\vec v\vec x^T,-1)$ for $\vec x\in \R^{n}$.
\end{proposition}
\begin{proof}
	It can be easily proved that the last row of $\tril(L,-1)$ is equal to
	$$v_n\cdot \begin{bmatrix}
			u_1 & \dots & u_{n-1} & 0
		\end{bmatrix}\begin{bmatrix}
			L_{n-1}^{-T} & \\&1
		\end{bmatrix},$$
	where $L_{n-1}$ is the leading principal submatrix of $L$ of size $n-1$. Using this fact the thesis can be proved recursively.
\end{proof}

\begin{remark}\label{rmk:nonzero_v}
	Notice that if $\triu(J_k,1)=\triu(\vec u \vec v^T, 1)$, the vector $\vec v$ cannot have zero entries: indeed if there exists $s\le k$ such that $v_s=0$, then, as a consequence of the proof of Theorem \ref{thm:semiseparable-projection}, the matrix $J_s-\diag(\gamma,\xi_1,\dots,\xi_{s-1})$ has the last column equal to zero for all $\gamma\in \R$, but this is impossible since for an appropriate choice of~$\gamma$ this matrix has to be invertible.
\end{remark}

Exploiting the quasiseparable structure of the matrix $B_k$, we can compute the matrices $B_k$ and $P_k$ by only performing a few scalar products. Indeed, if we let

\begin{equation*}
	B_k=\begin{bmatrix}
		d_1 & \beta_1  & \gamma_1                                         \\
		    & d_2      & \beta_2  & \gamma_2 & \bignonz                   \\
		    &          & \ddots   & \ddots   & \ddots                     \\
		    & \bigzero &          & d_{k-2}  & \beta_{k-2} & \gamma_{k-2} \\
		    &          &          &          & d_{k-1}     & \beta_{k-1}  \\
		    &          &          &          &             & d_{k}        \\
	\end{bmatrix},
\end{equation*}
and we define $\vec x_k = [P_{k-1} \, \vec 0 ] B_k \vec e_k$, we have that
\begin{equation}\label{eqn:structured-representation_B}
	A\vec q_k = AQ_k \vec e_k = P_kB_k \vec e_k = d_k\vec p_k + \vec x_k.
\end{equation} 

Using the fact that the submatrix of $B_k$ that involves the last two columns and all except for the last two rows has rank at most 1, we can compute $\vec x_k$ with the recursive relation 
\begin{equation*}
\vec x_k = \dfrac{\gamma_{k-2}}{\beta_{k-2}} \vec x_{k-1} + \beta_{k-1} \vec p_{k-1}.
\end{equation*}

 This allows us to compute $d_k, \beta_{k-1}, \gamma_{k-2}$ and $\vec p_k$ with only two scalar products. The $k$-th step of the procedure is summarized in Algorithm~\ref{algorithm:rational-Golub-Kahan}.

\begin{algorithm}
	\KwInput{$A,\vec q_k, \vec p_{k-1}, \vec p_{k-2}, \vec x_{k-1};$}
	\KwOutput{ $\vec p_k, d_k, \beta_{k-1}, \gamma_{k-2}, \vec x_k;$}
	$\vec w \gets A\vec q_k$\\
	$\beta_{k-1} \gets \vec w^T\vec p_{k-1}$\\
	$\gamma_{k-2} \gets \vec w^T\vec p_{k-2}$\\
	$\vec x_k \gets \frac{\gamma_{k-2}}{\beta_{k-2}}\vec x_{k-1}+\beta_{k-1}\vec p_{k-1}$;\\
	$\vec w \gets \vec w-\vec x_k$;\\
	$d_k\gets\norm{\vec w}_2$\\
	$\vec p_k=\frac{\vec w}{d_k}$;
	\caption{$k$-th step of rational Golub-Kahan algorithm} \label{algorithm:rational-Golub-Kahan}
\end{algorithm}

Notice that during the $k$-th step of the procedure, we do not require the first $k-1$ columns of $Q_k$. Moreover, for the computation of the projected solution $\barvec y_k$ defined in \eqref{eqn:gmf-approx-rational}, we do not need the matrix $Q_k$. For this reason, the computation of the matrix $Q_k$ can be performed by using a short recurrence rational Lanczos algorithm, as the one presented in
\cite[Section~5.2]{GuettelThesis}, and we can keep in memory only the last two columns of $Q_k$.
After the $k$-th step of the algorithm, the $k$-th column of the matrix $B_k$ can be computed using the newly computed quantities and the previous column, by exploting the rank structure of the matrix $B_k$.

Algorithm~\ref{algorithm:rational-Golub-Kahan} reduces to the standard Golub-Kahan bidiagonalization if all the poles are chosen equal to infinity.

\begin{remark}
	In the algorithm it is implicitly assumed that $\beta_i\neq 0$ for each $i$. In practice this hypothesis is always satisfied, however, as observed in Remark~\ref{rmk:nonzero_v}, if $k$ is less than the invariance index there is at least one nonzero off-diagonal entry in the $k$-th column of $B_k$. Hence we could modify the algorithm to avoid the issue of $\beta_i=0$.
\end{remark}
\begin{remark}The algorithm presented in this section also works if some of the poles are equal to zero. However, the proof of this fact requires slightly different tools, and hence we omitted it for brevity.
\end{remark}

\section{Error bounds}
\label{sec:error-bounds}
In this section we prove some error bounds for the approximation of $f^\diamond(A) \vec b$ using the polynomial and rational Krylov methods described above. These bounds link the approximation error with the error of polynomial and rational approximation of $f$ on an interval containing the singular values of $A$. Our results are the analogue of the ones that hold for standard matrix functions, and they can be proved in a similar way.

We first find an upper and lower bound for the singular values of $B_k$. For convenience, given a matrix $A \in \R^{m \times n}$, throughout this section we are going to use an extended notation for singular values, defining $\sigma_j := 0$ for all $j$ such that $\min\{m,n\} < j \le \max \{m,n\}$.

\begin{lemma}
	\label{lemma:interlacing-singular-values}
	Let $\sigma_1$ and $\sigma_n$ be the first and $n$-th singular value of $A \in \R^{m \times n}$, respectively. Then the singular values of $B_k$ belong to the interval $[ \sigma_n, \sigma_1 ]$.
\end{lemma}
\begin{proof}
	Let $\eta_1$ and $\eta_k$ be the largest and smallest singular value of $B_k$, respectively. Using the variational characterization of eigenvalues, we have
	\begin{equation*}
		\eta_k^2 = \min_{\substack{\vec v \in \R^k \\ \norm{\vec v}_2 = 1}} \vec v^T B_k^T B_k \vec v = \min_{\substack{\vec v \in \R^k \\ \norm{\vec v}_2 = 1}} \vec v^T Q_k^T A^T P_k P_k^T A Q_k \vec v.
	\end{equation*}
	Since the columns of $Q_k$ are an orthonormal basis of the subspace $\rat_k(A^TA, \vec b)$, we get
	\begin{equation*}
		\eta_k^2 = \min_{\substack{\vec w \in \rat_k(A^TA, \vec b) \\ \norm{\vec w}_2 = 1}} \vec w^T A^T P_k P_k^T A \vec w = \min_{\substack{\vec w \in \rat_k(A^TA, \vec b) \\ \norm{\vec w}_2 = 1}} \vec w^T A^T A \vec w,
	\end{equation*}
	where for the last equality we used the fact that $A\vec w \in \rat_k(AA^T, \vec A \vec b)$ and therefore $P_k P_k^T A \vec w = A \vec w$. Finally, we obtain
	\begin{equation*}
		\eta_k^2 = \min_{\substack{\vec w \in \rat_k(A^TA, \vec b) \\ \norm{\vec w}_2 = 1}} \vec w^T A^T A \vec w \ge \min_{\substack{\vec w \in \R^n \\ \norm{\vec w}_2 = 1}} \vec w^T A^T A \vec w = \lambda_{\min} (A^TA) = \sigma_n^2.
	\end{equation*}
	With the same procedure, we also get
	\begin{equation*}
		\eta_1^2 = \max_{\substack{\vec w \in \rat_k(A^TA, \vec b) \\ \norm{\vec w}_2 = 1}} \vec w^T A^T A \vec w \le \max_{\substack{\vec w \in \R^n \\ \norm{\vec w}_2 = 1}} \vec w^T A^T A \vec w = \lambda_{\max} (A^TA) = \sigma_1^2.
	\end{equation*}
\end{proof}

Observe that if $A \in \R^{m \times n}$ is rectangular with $n >m$, we always have $\sigma_n = 0$, and hence $B_k$ may have singular values arbitrarily close to $0$ even if $\sigma_{\min\{m,n\}}(A) > 0$.
This fact is going to affect the error bounds in Theorem~\ref{thm:polynomial-krylov-error-bound} and Theorem~\ref{thm:rational-krylov-error-bound}.

As an example, consider the $1 \times 2$ matrix $A = \begin{bmatrix}
		1 & 0 \\
	\end{bmatrix}$ and the vector $\vec b = \begin{bmatrix}
		\epsilon \\
		1
	\end{bmatrix}$, for small $\epsilon > 0$. For $k = 1$, we have $Q_1 = \vec b /\norm{ \vec b}_2 = \frac{1}{\sqrt{1 + \epsilon^2}} \vec b$, and $P_1 = A \vec b / \norm{A \vec b}_2 = 1$. So we have $B_1 = P_1^T A Q_1 = \epsilon \in \R^{1 \times 1}$, and hence $B_1$ can have an arbitrarily small singular value even if $\sigma_1(A) = 1$.

\subsection{Polynomial error bounds}
\label{subsubsec:pol-err-bounds}
We first prove the error bounds in the polynomial case. Recall that a polynomial Krylov method computes an approximation to $\vec y = f^\diamond(A) \vec b$ from the subspace $\kryl_k(A A^T, A \vec b)$ as
\begin{equation}
	\label{eqn:gmf-approx-polynomial--bounds}
	\barvec y_k = P_k f^\diamond(B_k) Q_k^T \vec b = \norm{\vec b}_2 P_k f^\diamond(B_k) \vec e_1,
\end{equation}
where $B_k = P_k^T A Q_k$, and $P_k$ and $Q_k$ are the matrices computed in the Golub-Kahan bidiagonalization of $A$, satisfying $\vspan(P_k) = \kryl_k(A A^T, A \vec b)$ and $\vspan(Q_k) = \kryl_k(A^T A, \vec b)$.

A key observation for proving the bounds is the exactness of the approximation~\eqref{eqn:gmf-approx-polynomial--bounds} when $f$ is an odd polynomial, stated in the following lemma.

\begin{lemma}
	\label{lemma:polynomial-krylov-exactness}
	Assume that $f = p_{2k-1}$ is an odd polynomial of degree $\le 2k-1$. Then the approximation $\barvec y_k$ given by \eqref{eqn:gmf-approx-polynomial--bounds} is exact, i.e., it holds $\vec y = \barvec y_k$.
\end{lemma}
\begin{proof}
	It is sufficient to prove this for $f(z) = z^{2\ell -1}$, $\ell = 1, \dots, k$. For $\ell = 1$, we have
	\begin{equation*}
		\vec y = A \vec b \quad \text{and} \quad \barvec y_k = P_k (P_k^T A Q_k) Q_k^T\vec b = A \vec b,
	\end{equation*}
	since $Q_kQ_k^T \vec b = \vec b$, and $P_k P_k^T A \vec b = A \vec b$, for all $k \ge 1$.

	For $\ell > 1$, we have by Proposition~\ref{prop:gmf-polynomial-evaluation} that
	\begin{equation}
		\label{eqn:proof-polynomial-exactness-1}
		\vec y = f^\diamond(A) = (AA^T)^{\ell-1} A \vec b
	\end{equation}
	and
	\begin{equation}
		\label{eqn:proof-polynomial-exactness-2}
		\barvec y_k = P_k f^\diamond(B_k) \vec e_1 = P_k (B_k B_k^T)^{\ell-1} B_k Q_k^T \vec b.
	\end{equation}
	Recalling the definition of $B_k$, we get
	\begin{equation*}
		\barvec y_k = P_k (B_kB_k^T)^{\ell-2} (P_k ^T A Q_k Q_k^T A^T P_k)(P_k^T A Q_k) Q_k^T \vec b.
	\end{equation*}
	We have $Q_k Q_k^T \vec b = \vec b$ since $\vec b \in \kryl_1(A^TA, \vec b)$, and similarly $P_kP_k^T A \vec b = A \vec b$ since $ A \vec b \in \kryl_1(AA^T, A \vec b)$. Hence, letting $\vec b_2 := A^TA \vec b \in \kryl_2(A^TA, \vec b)$, we obtain
	\begin{equation*}
		\barvec y_k = P_k (B_k B_k^T)^{\ell-2} (P_k^T A Q_k) Q_k^T \vec b_2,
	\end{equation*}
	which is the same situation as before, with $\ell$ replaced by $\ell-1$, and $\vec b \in \kryl_1(A^TA, \vec b)$ replaced by $\vec b_2 \in \kryl_2(A^TA, \vec b)$.

	Since $k \ge \ell$, we can repeat the above procedure until we are left with
	\begin{equation*}
		\barvec y_k = P_k (P_k^T A Q_k) Q_k^T \vec b_\ell, \qquad \text{where} \quad \vec b_\ell = (A^T A)^{\ell-1} \vec b \in \kryl_\ell(A^TA, \vec b).
	\end{equation*}
	Then, because $k \ge \ell$, we have that $Q_k Q_k^T \vec b_\ell = \vec b_\ell$, and likewise $P_kP_k^T A \vec b_\ell = A \vec b_\ell$, since $A \vec b_\ell \in \kryl_\ell(AA^T, A \vec b)$, so we obtain $\barvec y_k = A (A^TA)^{\ell-1} \vec b$. By comparing with \eqref{eqn:proof-polynomial-exactness-1}, we finally get $\vec y = \barvec y_k$.
\end{proof}

Using Lemma~\ref{lemma:polynomial-krylov-exactness}, we can prove the following theorem.
\begin{theorem}
	\label{thm:polynomial-krylov-error-bound}
	Let $A \in \R^{m \times n}$, and let $\sigma_1$, $\sigma_n$ and $\sigma_m$ be the first, $n$-th and $m$-th singular value of $A$, respectively. Let $\barvec y_k$ be the approximation to $\vec y = f^\diamond(A) \vec b$ given by \eqref{eqn:gmf-approx-polynomial--bounds}. Then the following inequality holds:
	\begin{equation}
		\label{eqn:gmf-polynomial-bound-1}
		\norm{f^\diamond(A) \vec b - \barvec y_k}_2 \le 2 \norm{\vec b}_2 \min_{p \in \poly_{k-1}} \norm{ f(z) - p(z^2)z }_{\infty, [\sigma_n, \sigma_1]}.
	\end{equation}
	Moreover, if $A$ is square with $\sigma_m = \sigma_n > 0$, or if $\displaystyle\lim_{z \to 0}\tfrac{f(z)}{z} = 0$, it also holds that
	\begin{equation}
		\label{eqn:gmf-polynomial-bound-2}
		\norm{f^\diamond(A) \vec b - \barvec y_k}_2 \le 2 \norm {A \vec b}_2 \min_{p \in \poly_{k-1}} \norm{f(z)/z - p(z^2)}_{\infty, [\sigma_{\max\{m,n\}}, \sigma_1]}.
	\end{equation}
\end{theorem}
\begin{proof}
	Let $p$ be a polynomial of degree $\le k-1$. Then $p_{2k-1}(z) = p(z^2) z$ is an odd polynomial of degree $\le 2k-1$, and by Lemma~\ref{lemma:polynomial-krylov-exactness} we have
	\begin{equation}
		\label{eqn:proof-polynomial-bound-1}
		p^\diamond_{2k-1}(A) \vec b = P_k p^\diamond_{2k-1}(B_k) Q_k^T \vec b.
	\end{equation}
	By adding and subtracting the quantity in~\eqref{eqn:proof-polynomial-bound-1} to $f^\diamond(A) \vec b - \barvec y_k$, we get
	\begin{equation}
		\label{eqn:proof-polynomial-bound-2}
		f^\diamond(A) \vec b - \barvec y_k = [f^\diamond(A) - p_{2k-1}^\diamond(A)] \vec b - P_k[f^\diamond(B_k) - p_{2k-1}^\diamond(B_k)]Q_k^T \vec b.
	\end{equation}
	By invariance of the $2$-norm under unitary transformations, we have
	\begin{equation*}
		\norm {f^\diamond(A) - p^\diamond_{2k-1}(A) }_2 = \norm{f - p_{2k-1}}_{\infty, \sigma_\text{sing}(A)} \le \norm{f - p_{2k-1}}_{\infty, [\sigma_{\min\{m,n\}}, \sigma_1]},
	\end{equation*}
	and similarly, by Lemma~\ref{lemma:interlacing-singular-values} it holds
	\begin{equation*}
		\norm {f^\diamond(B_k) - p^\diamond_{2k-1}(B_k) }_2 \le \norm{f - p_{2k-1}}_{\infty, [\sigma_n, \sigma_1]}.
	\end{equation*}
	Combining the above inequalities with \eqref{eqn:proof-polynomial-bound-2}, we get
	\begin{equation*}
		\norm{f^\diamond(A) \vec b - \barvec y_k}_2 \le 2 \norm{\vec b}_2 \norm{f - p_{2k-1}}_{\infty, [\sigma_n, \sigma_1]},
	\end{equation*}
	and by taking the minimum over $p \in \poly_{k-1}$ we obtain \eqref{eqn:gmf-polynomial-bound-1}.

	To prove \eqref{eqn:gmf-polynomial-bound-2}, recall that if $\sigma_m > 0$ or $\displaystyle \lim_{z \to 0} \tfrac{f(z)}{z} = 0$, by Proposition~\ref{prop:gmf-general-equivalence} we have $f^\diamond(A) = g(AA^T)A$, where $g(z) = f(\sqrt{z})/\sqrt{z}$, and similarly, if $\sigma_n > 0$ or $\displaystyle \lim_{z \to 0} \tfrac{f(z)}{z} = 0$ it holds $f^\diamond(B_k) = g(B_k B_k^T)B_k$.
	Therefore, by also using Proposition~\ref{prop:gmf-polynomial-evaluation}, we can rewrite \eqref{eqn:proof-polynomial-bound-2} in the form
	\begin{equation*}
		f^\diamond(A) \vec b - \barvec y_k = [ g(AA^T) - p(AA^T) ] A \vec b - P_k [g(B_k B_k^T) - p(B_k B_k^T)] B_k Q_k^T \vec b.
	\end{equation*}
	Given that the eigenvalues of $B_k B_k^T$ are the squares of the singular values of $B_k$, with a similar argument as before we obtain
	\begin{align*}
		\norm{f^\diamond(A) - \barvec y_k}_2 & \le \norm{A \vec b}_2 \left(\norm{g - p}_{\infty, [\sigma_m, \sigma_1]} + \norm{g - p}_{\infty, [\sigma_n^2, \sigma_1^2]} \right) \\
		                                     & \le 2 \norm{A \vec b}_2 \norm*{f(z)/z - p(z^2)}_{\infty, [\sigma_{\max\{m,n\}}, \sigma_1]}.
	\end{align*}
	As before, \eqref{eqn:gmf-polynomial-bound-2} follows by taking the minimum over $p \in \poly_{k-1}$.
\end{proof}

\begin{remark}
	\label{rem:rectang-discuss-post-poly-bound}
	Observe that if the matrix $A$ is not square, we have $\sigma_{\max\{m,n\}} = 0$, and hence the bound~\eqref{eqn:gmf-polynomial-bound-2} always involves a polynomial approximation over the whole interval $[0, \sigma_1]$, even when $\sigma_{\min\{m,n\}} > 0$. If $A \in \R^{m \times n}$ is rectangular with $m < n$, then the bound~\eqref{eqn:gmf-polynomial-bound-1} also involves the whole interval $[0, \sigma_1]$.

	A possible strategy to overcome this issue is to use Proposition~\ref{prop:link-gmf-a-and-gmf-at} and write
	\begin{equation*}
		\vec y = f^\diamond(A) \vec b = (A^+)^T f^\diamond(A^T) A \vec b.
	\end{equation*}
	The vector $\vec w = f^\diamond(A^T) A \vec b$ can be approximated using a Krylov method on $A^T$, and then $\vec y$ can be recovered by solving the least squares problem
	\begin{equation}
	\label{eqn:gmf-approx-transp}
		\vec y = (A^+)^T \vec w = \arg\min_{\vec y} \norm{A^T \vec y - \vec w}_2.
	\end{equation}
	By rewriting the problem in this form, if $m < n$ and $\sigma_m >0$ we get a bound involving approximation on the smaller interval $[\sigma_m, \sigma_1]$ for the approximation of~$\vec w$, which translates to a bound for the approximation of $\vec y$.
\end{remark}

The bound \eqref{eqn:gmf-polynomial-bound-1} can be manipulated to obtain a more explicit bound. Assume that $\sigma_n > 0$, and let $I = [\sigma_n, \sigma_1]$. The polynomial $p(z^2)z$ is odd, and we can assume that $f$ is also odd, so we have
\begin{align*}
	\min_{p \in \poly_{k-1}}\norm{f(z) - p(z^2)z}_{\infty, I} & = \min_{p \in \poly_{k-1}}\norm{f(z) - p(z^2)z}_{\infty, (-I) \cup I} \\
	                                                          & = \min_{q \in \poly_{2k-1}}\norm{f(z) - q(z)}_{\infty, (-I) \cup I},
\end{align*}
where we used the fact that the polynomial of best approximation on $(-I) \cup I$ for an odd function is itself odd.

Bounds on the asymptotic rate of convergence for the polynomial approximation of a function on the union of disjoint intervals $(-I) \cup I$ have been developed in~\cite{ChuiHasson83}.
Using \cite[Theorem~8.2]{Trefethen13} in the proof of \cite[Theorem~1]{ChuiHasson83} to get explicit inequalities instead of asymptotic bounds, we obtain the following quantitative version of \cite[Theorem~1]{ChuiHasson83}.

We denote by $E_\rho$ the ellipse with foci at $\pm 1$ and vertices at $\pm \frac{1}{2}(\rho + \frac{1}{\rho})$, and by $\tilde E_\rho$ its image under the linear function that maps $[-1, 1]$ to the interval $[a^2, b^2]$. The ellipse $\tilde E_\rho$ has foci at $a^2$ and $b^2$, and vertices at $\frac{1}{2}(a^2 + b^2) \pm \frac{1}{4}(\rho + \frac{1}{\rho})(b^2 - a^2)$.

\begin{theorem}
	\label{thm:chui-hasson-quantitative}
	Let $0 < a < b$ and let $I = [a, b]$. Assume that $f|_{-I}$ and $f|_{I}$ are, respectively, the restrictions of a function $f_1$ analytic in the left half-plane $\Real(z) < 0$, and of a function $f_2$ analytic in the right half-plane $\Real(z) > 0$. Then, for all $k > 0$ and $1 < \rho \le \dfrac{b+a}{b-a}$ it holds
	\begin{equation}
		\label{eqn:chui-hasson-bound-quantitative}
		\min_{p \in \poly_{2k-1}}\norm{f(z) - p(z)}_{\infty, (-I) \cup I} \le C \frac{\rho}{\rho - 1} \rho^{-k},
	\end{equation}
	where $C = C(\rho) = M_1 + M_2 + \frac{1}{a}(N_1 + N_2)$, with
	\begin{align*}
		M_1 & = \max_{-z \in \tilde E_\rho} \abs{f_1(\sqrt{-z})},  & N_1 & = \max_{-z \in \tilde E_\rho} \abs{f_1(\sqrt{-z})/\sqrt{-z}},  \\
		M_2 & = \max_{z \in \tilde E_\rho} \abs{f_2(\sqrt{z})}, & N_2 & = \max_{z \in \tilde E_\rho} \abs{f_2(\sqrt{z})/\sqrt{z}},
	\end{align*}
	where $\tilde E_\rho$ denotes the ellipse defined above. 
	\end{theorem}

\begin{remark}
If we take $\rho = \dfrac{b+a}{b-a}$, the ellipse $\tilde E_\rho$ has a vertex at $0$. In this case, depending on the function $f$, we may have $N_i = + \infty$ and hence $C = +\infty$. In such a situation the bound only makes sense for $\rho < \dfrac{b+a}{b-a}$.
\end{remark}

Assuming that $\sigma_n > 0$, by plugging \eqref{eqn:chui-hasson-bound-quantitative} in the bound \eqref{eqn:gmf-polynomial-bound-1}, we get
\begin{equation}
\label{eqn:gmf-polynomial-bound-chui-hasson}
	\norm{f^\diamond(A) \vec b - \barvec y_k}_2 \le 2 C \norm{\vec b}_2 \frac{\rho}{\rho - 1} \rho^{-k}, \qquad \text{for }1 < \rho \le \frac{\sigma_1 + \sigma_n}{\sigma_1 - \sigma_n},
\end{equation}
where $C$ is as defined in Theorem~\ref{thm:chui-hasson-quantitative}.

\subsection{Rational error bounds}
\label{subsubsec:rat-err-bounds}

Next, we prove similar error bounds for the rational approximation \eqref{eqn:gmf-approx-rational}. The strategy of the proof is the same as in the polynomial case, even though the proof itself is a bit more technical. We start by stating the result analogous to Theorem \ref{thm:polynomial-krylov-error-bound}. Recall that the denominator in the rational Krylov space $\rat_k(AA^T, A \vec b)$ is given by the polynomial $q_{k-1}(z) = \dprod_{j = 1}^{k-1} (1 - z/\xi_j)$, where $\{ \xi_j \}_{j \ge 1}$ is a sequence of poles in $(\C \cup \{\infty\}) \setminus \sigma(AA^T)$. For convenience, in this section we use as denominator polynomial for $\rat_k(AA^T, \vec b)$ the polynomial $\tilde q_{k-1}(z) = \dprod_{j = 1}^{k-1} (z - \xi_j)$. This polynomial differs from $q_{k-1}$ only by a multiplicative constant, and hence it identifies the same rational Krylov subspace. The results of this section, such as Theorem~\ref{thm:rational-krylov-error-bound}, can be equivalently stated in terms of $q_{k-1}$ or $\tilde q_{k-1}$.

\begin{theorem}
	\label{thm:rational-krylov-error-bound}
	Let $A \in \R^{m \times n}$, and let $\sigma_1$, $\sigma_n$ and $\sigma_m$ be the first, $n$-th and $m$-th singular value of $A$, respectively. Let $\barvec y_k$ be the approximation to $\vec y = f^\diamond(A) \vec b$ from $\rat_k(AA^T, A \vec b)$ given by \eqref{eqn:gmf-approx-rational}. Then the following inequality holds:
	\begin{equation}
		\label{eqn:gmf-rational-bound-1}
		\norm{f^\diamond(A) \vec b - \barvec y_k}_2 \le 2 \norm{\vec b}_2 \min_{p \in \poly_{k-1}} \norm{ f(z) - q_{k-1}(z^2)^{-1} p(z^2)z }_{\infty, [\sigma_n, \sigma_1]}.
	\end{equation}
	Moreover, if $A$ is square with $\sigma_m = \sigma_n > 0$, or if $\displaystyle \lim_{z \to 0} \tfrac{f(z)}{z} = 0$, it also holds that
	\begin{equation}
		\label{eqn:gmf-rational-bound-2}
		\norm{f^\diamond(A) \vec b - \barvec y_k}_2 \le 2 \norm {A \vec b}_2 \min_{p \in \poly_{k-1}} \norm{f(z)/z - q_{k-1}(z^2)^{-1} p(z^2)}_{\infty, [\sigma_{\max\{m,n\}}, \sigma_1]}.
	\end{equation}
\end{theorem}

Note that the same issues discussed after Theorem~\ref{thm:polynomial-krylov-error-bound} in the case of rectangular matrices also arise in the rational case, and the same approach proposed in Remark~\ref{rem:rectang-discuss-post-poly-bound} can be used to address them.

Similarly to the polynomial case, the bound \eqref{eqn:gmf-rational-bound-1} can be rewritten by exploiting the fact that $f$ can be assumed to be odd and hence that the best approximant on a symmetric interval is odd, yielding
\begin{equation}
\label{eqn:gmf-rational-bound--2-intervals}
	\norm{f^\diamond(A) \vec b - \barvec y_k}_2 \le 2 \norm{\vec b}_2 \min_{p \in \poly_{2k-1}} \norm{f(z) - q_{k-1}(z^2)^{-1} p(z)}_{\infty, (-I) \cup I},
\end{equation}
where again $I = [\sigma_n, \sigma_1]$. However, rational approximation on disjoint intervals is a complicated problem, and hence this formulation might be less useful in practice.

A more practical way to rewrite the bound~\eqref{eqn:gmf-rational-bound-1} is the following:
\begin{align}
	\nonumber
	\norm{f^\diamond(A) \vec b - \barvec y_k}_2 & \le 2 \norm{\vec b}_2 \min_{p \in \poly_{k-1}} \norm{ f(z) - q_{k-1}(z^2)^{-1} p(z^2)z }_{\infty, [\sigma_n, \sigma_1]}                                  \\
	\nonumber
	& = 2 \norm{\vec b}_2 \min_{p \in \poly_{k-1}} \norm{ \sqrt{z} \big(f(\sqrt{z})/\sqrt{z} - q_{k-1}(z)^{-1}p(z)\big) }_{\infty, [\sigma_n^2, \sigma_1^2]}   \\
	\label{eqn:gmf-rational-bound--practical}
	& \le 2 \sigma_1 \norm{\vec b}_2 \min_{p \in \poly_{k-1}} \norm{\big(f(\sqrt{z})/\sqrt{z} - q_{k-1}(z)^{-1}p(z)\big) }_{\infty, [\sigma_n^2, \sigma_1^2]}.
\end{align}
Although we get an additional factor $\sigma_1$, this bound relates the error with a uniform rational approximation problem on a real interval. This approximation problem is well-studied in the literature, and it is the same that appears when computing standard matrix functions with rational Krylov methods, so it can be a viable tool for selecting good poles.

Before we prove Theorem~\ref{thm:rational-krylov-error-bound}, we state and prove a few auxiliary lemmas. As in the polynomial case, the key point for the proof is the exactness of the rational approximation on functions of the form $f(z) = q_{k-1}(z^2)^{-1} p(z^2) z$, where $p$ is any polynomial in $\poly_{k-1}$. In order to prove this fact, we are going to use the following lemma, which allows us to replace $B_k$ with $A$ as we did in the proof of Lemma~\ref{lemma:polynomial-krylov-exactness}.

\begin{lemma}
	\label{lemma:replace-Bk-with-A-rational}
	Let $k \ge j+1$, and let $\vec v \in \rat_j(A^TA, \vec b)$, with $\tilde q_{k-1}(z) = \dprod_{i = 1}^{k-1}(z - \xi_i)$. Then:
	\begin{align}
		\label{eqn:replace-Bk-with-A-1}
		(B_k^TB_k - \xi_j I)^{-1} Q_k^T \vec v           & = Q_k^T (A^TA - \xi_{j} I )^{-1} \vec v    \\
		\label{eqn:replace-Bk-with-A-2}
		B_k^T B_k (B_k^TB_k - \xi_j I)^{-1} Q_k^T \vec v & = Q_k^T A^TA (A^TA - \xi_j I )^{-1} \vec v
	\end{align}
\end{lemma}
\begin{proof}
	To prove \eqref{eqn:replace-Bk-with-A-1}, it is enough to show that
	\begin{equation*}
		(B_k^TB_k - \xi_j I) Q_k^T (A^TA - \xi_j I)^{-1} \vec v = Q_k^T \vec v.
	\end{equation*}
	Letting $\vec x = (A^TA - \xi_j I)^{-1} \vec v \in \rat_{j+1}(A^TA, \vec b)$ and observing that $Q_k Q_k^T \vec x = \vec x$, we have
	\begin{align*}
		(B_k^TB_k - \xi_j I) Q_k^T (A^TA - \xi_j I)^{-1} \vec v & = Q_k^T(A^T P_k P_k^TA - \xi_j I) Q_k Q_k^T \vec x \\
		                                                        & = Q_k^T(A^T P_k P_k^TA \vec x - \xi_j \vec x)      \\
		                                                        & = Q_k^T(A^T A - \xi_j I )\vec x = Q_k^T \vec x,
	\end{align*}
	where we also used that $P_k P_k^T A \vec x = A \vec x$, since $A \vec x \in \rat_{j+1}(AA^T, A \vec b)$. This proves \eqref{eqn:replace-Bk-with-A-1}.

	To prove \eqref{eqn:replace-Bk-with-A-2}, by using~\eqref{eqn:replace-Bk-with-A-1} we have
	\begin{align*}
		B_k^T B_k (B_k^TB_k - \xi_j I)^{-1} Q_k^T \vec v & = B_k^T B_k Q_k^T \vec x                  \\
		                                                 & = Q_k^T A P_k P_k^T A^T Q_k Q_k^T \vec x.
	\end{align*}
	Now, observe that $\vec x \in \rat_{j+1}(A^TA, \vec b)$ and that $A \vec x \in \rat_{j+1}(AA^T, A \vec x)$, and hence it holds $Q_kQ_k^T \vec x = \vec x$ and $P_kP_k^T A \vec x = A \vec x$. With these facts, we get
	\begin{equation*}
		B_k^T B_k (B_k^TB_k - \xi_j I)^{-1} Q_k^T \vec v = Q_k^T A A^T \vec x,
	\end{equation*}
	which is equivalent to \eqref{eqn:replace-Bk-with-A-2}.
\end{proof}

\begin{lemma}
	\label{lemma:rational-krylov-exactness}
	Assume that $f$ is of the form $f(z) = q_{k-1}(z^2)^{-1} p(z^2)z$, where $p \in \poly_{k-1}$. Then the approximation $\barvec y_k$ from the rational Krylov subspace $\rat_k(AA^T, A \vec b)$ given by \eqref{eqn:gmf-approx-rational} is exact, i.e., it holds $\vec y = \barvec y_k$.
\end{lemma}
\begin{proof}
	It is sufficient to show that $\vec y = \barvec y_k$ for $f(z) = \dfrac{z^{2\ell - 1}}{\tilde q_{k-1}(z^2)}$, for $\ell = 1, \dots, k$. By Corollary \ref{cor:gmf-rational-evaluation}, we have
	\begin{equation*}
		\vec y = f^\diamond(A) \vec b = A (A^TA)^{\ell - 1} \tilde q_{k-1}(A^TA)^{-1} \vec b.
	\end{equation*}
	Similarly, we have
	\begin{align*}
		\barvec y_k & = P_k f^\diamond(B_k) Q_k^T \vec b = P_k B_k (B_k^T B_k)^{\ell - 1} \tilde q_{k-1}(B_k^T B_k)^{-1} Q_k^T \vec b                                    \\
		            & = P_k B_k \prod_{j=\ell}^{k-1} (B_k^T B_k - \xi_j I)^{-1} \prod_{j = 1}^{\ell-1}\left [(B_k^T B_k)(B_k^T B_k - \xi_j I)^{-1}\right ] Q_k^T \vec b.
	\end{align*}
	Now, define the vectors $\vec t_m \in \C^k$ as
	\begin{equation*}
		\vec t_m = \begin{cases}
			\dprod_{j = 1}^{m-1}\left [(B_k^T B_k)(B_k^T B_k - \xi_j I)^{-1}\right ] Q_k^T \vec b.                                                  & \text{if $1 \le m \le \ell$,} \\
			\dprod_{j=\ell}^{m-1} (B_k^T B_k - \xi_j I)^{-1}\prod_{j = 1}^{\ell-1}\left [(B_k^T B_k)(B_k^T B_k - \xi_j I)^{-1}\right ] Q_k^T \vec b & \text{if $\ell < m \le k$}.
		\end{cases}
	\end{equation*}
	It is straightforward to see that $\vec t_{m+1} = B_k^T B_k (B_k^T B_k - \xi_m I)^{-1} \vec t_m$ for $1 \le m < \ell$, and that $\vec t_{m+1} = (B_k^T B_k - \xi_m I)^{-1} \vec t_m$ for $\ell \le m < k$.
	By Lemma~\ref{lemma:replace-Bk-with-A-rational}, we have
	\begin{equation*}
		\vec t_2 = B_k^T B_k (B_k^TB_k - \xi_1 I)^{-1} Q_k^T \vec b = Q_k^T A^TA (A^TA - \xi_1 I)^{-1} \vec b,
	\end{equation*}
	and hence $\vec t_2 = Q_k^T \vec b_2$, with $\vec b_2 \in \rat_2(A^TA, \vec b)$.

	By induction, assume that $\vec t_m = Q_k^T \vec b_m$, where $\vec b_m \in \rat_m(A^TA, \vec b)$. Then, if $m < \ell$, we have
	\begin{equation*}
		\vec t_{m+1} = B_k^TB_k(B_k^T B_k - \xi_m I)^{-1} Q_k^T \vec b_m = Q_k^T A^TA(A^TA - \xi_m I)^{-1}\vec b_m,
	\end{equation*}
	where we used \eqref{eqn:replace-Bk-with-A-2}. Defining $\vec b_{m+1} = A^TA(A^TA - \xi_m I)^{-1}\vec b_m$, we get that $\vec t_{m+1} = Q_k^T \vec b_{m+1}$, where $\vec b_{m+1} \in \rat_{m+1}(A^TA, \vec b)$. The case $\ell \le m < k$ is similar: by using \eqref{eqn:replace-Bk-with-A-1}, we find that $\vec t_{m+1} = Q_k^T \vec b_{m+1}$, with $\vec b_{m+1} = (A^TA - \xi_m I)^{-1} \vec b_m \in \rat_{m+1}(A^TA, \vec b)$.

	Considering $m = k$, by the above discussion we have obtained that
	\begin{align*}
		\vec t_k = Q_k^T \vec b_k & = Q_k^T (A^TA)^{\ell-1} \dprod_{j = 1}^{k-1}(A^TA - \xi_j I)^{-1} \vec b \\
		                          & = Q_k^T (A^TA)^{\ell-1} \tilde q_{k-1}(A^TA)^{-1} \vec b,
	\end{align*}
	with $\vec b_k = (A^TA)^{\ell-1} \tilde q_{k-1}(A^TA)^{-1} \vec b \in \rat_k(A^TA, \vec b)$. Therefore we have
	\begin{equation*}
		\barvec y_k = P_k B_k \vec t_k = P_k P_k^T A Q_k Q_k^T \vec b_k = A \vec b_k,
	\end{equation*}
	where we used $Q_kQ_k^T \vec b_k = \vec b_k$ and $P_kP_k^T A \vec b_k = A \vec b_k$. Hence we conclude that $\vec y = \barvec y_k$.
\end{proof}

We now have all the elements to prove Theorem~\ref{thm:rational-krylov-error-bound}. The proof follows the same strategy as the proof of Theorem~\ref{thm:polynomial-krylov-error-bound}.

\begin{proof}[Proof of Theorem~\ref{thm:rational-krylov-error-bound}]
	Let $p$ be a polynomial of degree $\le k-1$. Then the rational function $r(z) = \tilde q_{k-1}(z^2)^{-1}p(z^2) z$ satisfies the assumptions of Lemma~\ref{lemma:rational-krylov-exactness}, and hence we have
	\begin{equation}
		\label{eqn:proof-rational-bound-1}
		r^\diamond(A) \vec b = P_k r^\diamond(B_k) Q_k^T \vec b.
	\end{equation}
	By adding and subtracting the quantity in~\eqref{eqn:proof-rational-bound-1} to $f^\diamond(A) \vec b - \barvec y_k$, we get
	\begin{equation}
		\label{eqn:proof-rational-bound-2}
		f^\diamond(A) \vec b - \barvec y_k = [f^\diamond(A) - r^\diamond(A)] \vec b - P_k[f^\diamond(B_k) - r^\diamond(B_k)]Q_k^T \vec b.
	\end{equation}
	Since by Lemma~\ref{lemma:interlacing-singular-values} the nonzero singular values of $B_k$ are contained in the interval $[\sigma_n, \sigma_1]$, by invariance of the $2$-norm under unitary transformations, we have
	\begin{align*}
		\norm {f^\diamond(A) - r^\diamond(A) }_2     & \le \norm{f - r}_{\infty, [\sigma_{\min\{m,n\}}, \sigma_1]}, \\
		\norm {f^\diamond(B_k) - r^\diamond(B_k) }_2 & \le \norm{f - r}_{\infty, [\sigma_n, \sigma_1]}.
	\end{align*}
	Combining the above inequalities with \eqref{eqn:proof-rational-bound-2}, we get
	\begin{equation*}
		\norm{f^\diamond(A) \vec b - \barvec y_k}_2 \le 2 \norm{\vec b}_2 \norm{f - r}_{\infty, [\sigma_n, \sigma_1]},
	\end{equation*}
	and by taking the minimum over $p \in \poly_{k-1}$ we obtain \eqref{eqn:gmf-rational-bound-1}.

	To prove \eqref{eqn:gmf-rational-bound-2}, if $\sigma_n = \sigma_m > 0$ or $\displaystyle\lim_{z \to 0}\tfrac{f(z)}{z} = 0$, by Proposition~\ref{prop:gmf-general-equivalence} we can write $f^\diamond(A) = g(AA^T)A$ and $f^\diamond(B_k) = g(B_kB_k^T) B_k$, where $g(z) = f(\sqrt{z})/\sqrt{z}$.
	Thus, using also Corollary~\ref{cor:gmf-rational-evaluation}, we can rewrite \eqref{eqn:proof-rational-bound-2} in the form
	\begin{equation*}
		f^\diamond(A) \vec b - \barvec y_k = h(AA^T) A \vec b - P_k h(B_kB_k^T) B_k Q_k^T \vec b,
	\end{equation*}
	where $h(z) = g(z) - \tilde q_{k-1}(z)^{-1}p(z)$.

	Given that the eigenvalues of $B_k B_k^T$ are the squares of the singular values of $B_k$, using Lemma~\ref{lemma:interlacing-singular-values} and proceeding as above we obtain
	\begin{align*}
		\norm{f^\diamond(A) - \barvec y_k}_2 & \le \norm{A \vec b}_2 \left( \norm{h}_{\infty, [\sigma_m^2, \sigma_1^2]} + \norm{h}_{\infty, [\sigma_n^2, \sigma_1^2]} \right) \\
		                                     & = 2 \norm{A \vec b}_2 \norm*{f(z)/z - \tilde q_{k-1}(z^2)^{-1} p(z^2)}_{\infty, [\sigma_{\max\{m,n\}}, \sigma_1]}.
	\end{align*}
	As before, \eqref{eqn:gmf-rational-bound-2} follows by taking the minimum over $p \in \poly_{k-1}$.
\end{proof}

We mention that the results of this section can also be obtained by exploiting the link betweeen GMFs of $A$ and standard functions of the matrix $\mathcal{A} = \begin{bmatrix}
		0   & A \\
		A^T & 0
	\end{bmatrix}$.
Indeed, it was observed in \cite{arrigo2016computation} that for an odd function $f$ it holds
\begin{equation*}
	f(\mathcal{A}) = \begin{bmatrix}
		0               & f^\diamond(A) \\
		f^\diamond(A^T) & 0
	\end{bmatrix}.
\end{equation*}
If we define the orthogonal matrix $\mathcal{U}_{2k} = \begin{bmatrix}
		P_k & 0   \\
		0   & Q_k
	\end{bmatrix}$ and the vector $\vec c = \begin{bmatrix}
		0 \\
		\vec b
	\end{bmatrix} \in \R^{m+n}$, we have that $\mathcal{U}_{2k}^T \mathcal{A} \mathcal{U}_{2k} = \begin{bmatrix}
		0     & B_k \\
		B_k^T & 0
	\end{bmatrix}$ and
\begin{align*}
	f(\mathcal{A}) \vec c                                                                         & = \begin{bmatrix}
		f^\diamond(A) \vec b \\
		0
	\end{bmatrix}, \\
	\mathcal{U}_{2k} f(\mathcal{U}_{2k}^T \mathcal{A} \mathcal{U}_{2k}) \mathcal{U}_{2k}^T \vec c & = \begin{bmatrix}
		P_k f^\diamond(B_k) Q_k^T \vec b \\
		0
	\end{bmatrix}.
\end{align*}
Moreover, it can be proved that the columns of $\mathcal{U}_{2k}$ are an orthonormal basis for a rational Krylov subspace $\rat_{2k}(\mathcal{A}, \vec c)$, the poles of which consist of a single pole at $\infty$, and $\pm \theta_j$, $j = 1, \dots, k-1$, where $\theta_j^2 = \xi_j$ for each $j$. This fact is straightforward to prove in the polynomial case, where all $\theta_j$ are equal to $\infty$.

An alternative derivation of the error bounds in Theorem \ref{thm:polynomial-krylov-error-bound} and Theorem \ref{thm:rational-krylov-error-bound} could then be obtained by combining the above fact with Lemma \ref{lemma:interlacing-singular-values} and error bounds concerning rational Krylov approximation of standard matrix functions (see, e.g., \cite[Corollary~3.4]{Guettel13}).

\subsection{An optimal pole for the Shift-and-Invert method}
\label{subsec:optimal-pole-SI}

In this section we use the error bounds in Theorem~\ref{thm:rational-krylov-error-bound} combined with a known result from approximation theory to find a pole that optimizes the bounds in the case of a single repeated pole (Shift-and-Invert method) located on the negative real line.

We consider the case of a nonsingular square matrix $A \in \R^{n \times n}$, with singular values contained in the interval $[\sigma_{\min}, \sigma_{\max}]$, with $\sigma_{\min} > 0$.

Note that with a change of variables the bound \eqref{eqn:gmf-rational-bound-2} can be rewritten in the form
\begin{equation*}
	\norm{f^\diamond(A) \vec b - \barvec y_k}_2 \le 2 \norm{A \vec b}_2 \min_{p \in \poly_{k-1}} \norm{ g(z) - q_{k-1}(z)^{-1} p(z) }_{\infty, [\sigma_{\min}^2, \sigma_{\max}^2]},
\end{equation*}
where the function $g$ is defined as $g(z) = f(\sqrt{z})/\sqrt{z}$.

In the case of a single repeated pole $\xi < 0$, we have $q_{k-1}(z) = (z - \xi)^{k-1}$ and
\begin{equation*}
	\Big\{ (z - \xi)^{-k+1} p(z) \,:\, p \in \poly_{k-1} \Big\} = \Big\{ p((z-\xi)^{-1}) \,:\, p \in \poly_{k-1} \Big\}.
\end{equation*}
By defining $h(z) = g(z^{-1} + \xi)$, so that $g(z) = h((z - \xi)^{-1})$, we get
\begin{equation}
	\label{eqn:min-polynomial--SI-pole}
	\min_{p \in \poly_{k-1}} \norm{g(z) - (z - \xi)^{-k+1}p(z)}_{\infty, [\sigma_{\min}^2, \sigma_{\max}^2]} = \min_{p \in \poly_{k-1}} \norm{h(z) - p(z)}_{\infty, [\mu_\text{min}, \mu_\text{max}]},
\end{equation}
where $\mu_\text{min} := (\sigma_{\max}^2 - \xi)^{-1}$ and $\mu_\text{max} := (\sigma_{\min}^2 - \xi)^{-1}$. Notice that for $\xi < 0$ it holds indeed $0 < \mu_\text{min} \le \mu_\text{max} \le (-\xi)^{-1}$.

The minimum in \eqref{eqn:min-polynomial--SI-pole} can be bounded using the following result in approximation theory, adapted from \cite[Proposition~3.1]{MoretNovati19}. Its proof relies on classical bounds for Faber series, see~\cite[Corollary~2.2]{Ellacott83}.

\begin{proposition}
	\label{prop:polynomial-bound--SI-pole}
	Let $\xi < 0$, and assume that $h(z) = g(z^{-1} + \xi)$ is analytic in the strip $0 < \Real z < (-\xi)^{-1}$ and continuous in $[0, (-\xi)^{-1}]$. Then, for any integer $k \ge 1$ it holds
	\begin{equation}
		\label{eqn:SI-bound--SI-pole}
		\min_{p \in \poly_{k-1}} \norm{h(z) - p(z)}_{\infty, [\mu_{\min}, \mu_{\max}]} \le 2M \frac{\rho^k}{1-\rho},
	\end{equation}
	where $M = \norm{h(z)}_{\infty, [0, (-\xi)^{-1}]}$ and
	\begin{equation*}
		\rho = \max \left\{ \frac{\sqrt{\sigma_{\max}^2 - \xi} - \sqrt{\sigma_{\min}^2 - \xi}}{\sqrt{\sigma_{\max}^2 - \xi} + \sqrt{\sigma_{\min}^2 - \xi}}\, , \, \frac{\sigma_{\max} \sqrt{\sigma_{\min}^2 - \xi} - \sigma_{\min}\sqrt{\sigma_{\max}^2 - \xi}}{\sigma_{\max} \sqrt{\sigma_{\min}^2 - \xi} + \sigma_{\min}\sqrt{\sigma_{\max}^2 - \xi}} \right\}.
	\end{equation*}
\end{proposition}
It follows from the analysis after \cite[Proposition~3.1]{MoretNovati19} that the bound \eqref{eqn:SI-bound--SI-pole} is optimized by choosing $\xi = -\sigma_{\min} \sigma_{\max}$. This choice leads to the following bound for the Shift-and-Invert iterates:
\begin{equation}
	\label{eqn:SI-bound-chosen-pole--SI-pole}
	\norm{\vec y - \barvec y_k}_2 \le 2 \norm{\vec b}_2 M \sqrt{\frac{\sigma_\text{max}}{\sigma_\text{min\phantom i\!}}} \exp\Big(-2k \sqrt \frac{\sigma_\text{min}}{\sigma_\text{max\phantom i\!}}\Big).
\end{equation}

We remark that the original result (see \cite[equation~(3.4)]{MoretNovati19}) exihibited an error like $\displaystyle\exp\big(-2k \sqrt[4]{\frac{a}{b}}\big)$ for a symmetric matrix $A$ with spectrum in $[a, b] \subset (0, +\infty)$, when using the Shift-and-Invert method with the optimal pole $\xi = -\sqrt{a b}$. The fact that in \eqref{eqn:SI-bound-chosen-pole--SI-pole} we have $\sqrt{\dfrac{\sigma_\text{min}}{\sigma_\text{max\phantom i\!}}}$ instead of $\sqrt[4]{\dfrac{\sigma_\text{min}}{\sigma_\text{max\phantom i\!}}}$ is not surprising, since we are essentially applying the result from \cite{MoretNovati19} to the matrix $A^TA$, whose spectrum is contained in $[\sigma_\text{min}^2, \sigma_\text{max\phantom i\!}^2]$.

\section{Numerical results}
\label{sec:numerical}

In this section we present some numerical experiments with the purpose of illustrating the error bounds and comparing the different methods proposed in the previous sections. We test the methods on randomly generated matrices with a prescribed distribution of singular values, obtained by taking two random orthogonal matrices $U, V \in \R^{n \times n}$ and constructing $A = U \Sigma V^T$, where $\Sigma = \diag(\sigma_1, \dots, \sigma_n) \in \R^{n \times n}$.
The random orthogonal matrices are obtained by taking a matrix $B \in \R^{n \times n}$ with entries from the normal distribution $\normal(0,1)$ and computing the QR factorization $B = QR$. If the diagonal entries of $R$ are nonnegative, then $Q$ is a random orthogonal matrix from the Haar distribution, a natural uniform probability distribution on the manifold of $n \times n$ orthogonal matrices \cite{Stewart80}.

The experiments were done with MATLAB, using the \texttt{rat\_krylov} function from the Rational Krylov Toolbox \cite{RKToolbox} for the implementation of the rational Arnoldi algorithm.
The plots display the relative 2-norm error $\norm{f^\diamond(A) \vec b - \barvec y_k}_2$, where $\barvec y_k$ is the approximation defined in \eqref{eqn:gmf-approx-polynomial} or \eqref{eqn:gmf-approx-rational}, depending on the Krylov method that was used.

\subsection{Error bounds}

We start by illustrating in Figure~\ref{fig:polynomial-convergence} the sharpness of the bound \eqref{eqn:gmf-polynomial-bound-chui-hasson} for the polynomial Krylov method. Under the assumptions of Theorem~\ref{thm:chui-hasson-quantitative}, the rate of convergence in the bound only depends on the interval $[\sigma_n, \sigma_1]$, and hence we can expect it to be pessimistic for most functions.
Indeed, for entire functions such as $\sinh(z)$ and $\sin(z)$ (see Figure~\ref{fig:polynomial-convergence}(b)) the convergence is much faster than the bound~\eqref{eqn:gmf-polynomial-bound-chui-hasson}; however, the bound can capture the asymptotic rate of convergence for certain functions with lower regularity such as $\sqrt{z}$, for suitable singular value distributions (see Figure~\ref{fig:polynomial-convergence}(a)). This implies that, under the same assumptions of Theorem~\ref{thm:chui-hasson-quantitative}, it is only possible to improve the multiplicative constant in the bound~\eqref{eqn:gmf-polynomial-bound-chui-hasson}. Note that in Figure~\ref{fig:polynomial-convergence} the multiplicative constant in the bound~\eqref{eqn:gmf-polynomial-bound-chui-hasson} was ignored for better visualization.

\begin{figure}
	\makebox[\linewidth][c]{
			\begin{subfigure}[t]{.6\textwidth}
		\begin{tikzpicture}
			\begin{semilogyaxis}[
				title = {},  % whatever name you want
				xlabel = {Iteration},
				ylabel = {Error},
				x tick label style={/pgf/number format/.cd,%
					scaled x ticks = false,
					set thousands separator={},
					fixed},
				legend pos=north east,
				ymajorgrids=true,
				grid style=dashed,legend style={font=\tiny}]
				\addplot[color=black, mark size=1pt, dashed] table {polybound.dat};
				\addlegendentry{{Bound~(6.10)}}
				
				\addplot[color=blue, forget plot] table {polyerr_sqrtzlogz.dat};
				\addplot[only marks, color=blue,mark=*, mark size=1pt, each nth point=10, forget plot] table {polyerr_sqrtzlogz.dat};
				\addplot[color=blue,mark=*, mark size=1pt, each nth point=1000] table {polyerr_sqrtzlogz.dat};
				\addlegendentry{{$\sqrt{z}\log(z)$}}
				
				\addplot[color=red, forget plot] table {polyerr_z_m025.dat};
				\addplot[only marks, color=red,mark=square*, mark size=1pt, each nth point = 10, forget plot] table {polyerr_z_m025.dat};
				\addplot[color=red,mark=square*, mark size=1pt, each nth point = 1000] table {polyerr_z_m025.dat};
				\addlegendentry{{$z^{-1/4}$}}
				
				\addplot[color=green, forget plot] table {polyerr_sqrtz.dat};
				\addplot[only marks, color=green, mark=triangle*, mark size=1pt, each nth point = 10, forget plot] table {polyerr_sqrtz.dat};
				\addplot[color=green, mark=triangle*, mark size=1pt, each nth point = 1000] table {polyerr_sqrtz.dat};
				\addlegendentry{{$\sqrt{z}$}}
			\end{semilogyaxis}
	\end{tikzpicture}
		\caption{}
\end{subfigure}
		\begin{subfigure}[t]{.6\textwidth}
	\begin{tikzpicture}
			\begin{semilogyaxis}[
				title = {},  % whatever name you want
				xlabel = {Iteration},
				ylabel = {Error},
				legend pos=north east,
				ymajorgrids=true,
				grid style=dashed,legend style={font=\tiny}]
				\addplot[color=blue,mark=*, mark size=1pt] table {polyerr_sinh.dat};
				\addlegendentry{{$\sinh(z)$}}
				\addplot[color=red,mark=square*, mark size=1pt] table {polyerr_sin.dat};
				\addlegendentry{{$\sin(z)$}}
			\end{semilogyaxis}
	\end{tikzpicture}
\caption{}
\end{subfigure}}
\caption{Convergence of the polynomial Krylov method for the approximation of $f^\diamond(A) \vec b$, where $A$ is a $2000 \times 2000$ matrix whose singular values are Chebyshev points of the second kind for the interval $[10^{-1}, 10]$, and $\vec b$ is a random vector. Left: functions with an asymptotic convergence rate predicted by the bound~\eqref{eqn:gmf-polynomial-bound-chui-hasson}. Right: entire functions with fast convergence.}
\label{fig:polynomial-convergence}%
\end{figure}
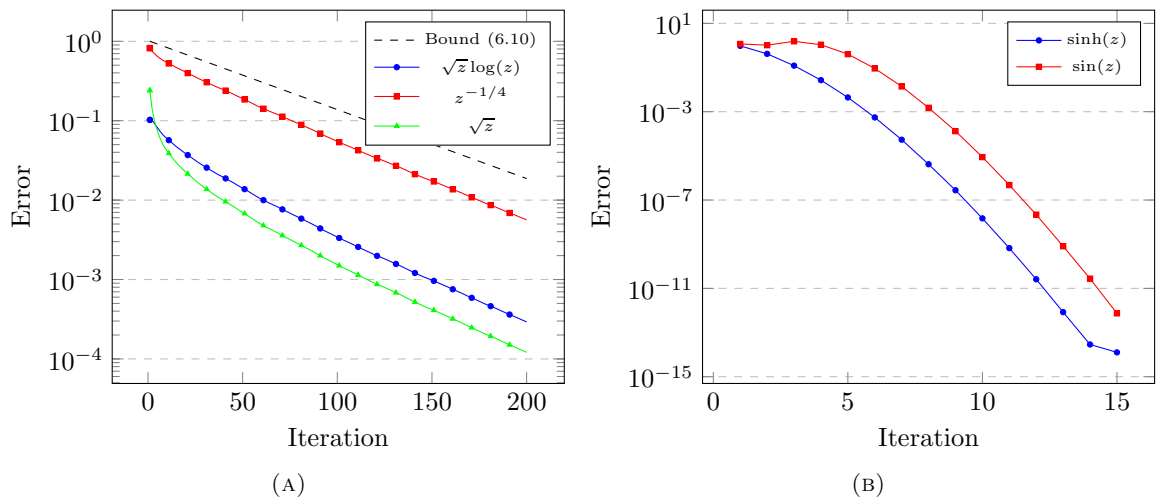

In Figure~\ref{fig:rational-convergence} we compare the convergence of the rational Krylov methods and we test the sharpness of the bounds~\eqref{eqn:gmf-rational-bound--2-intervals} and~\eqref{eqn:SI-bound-chosen-pole--SI-pole}. We use the Shift-and-Invert method with the pole $\xi = - \sigma_\text{min} \sigma_\text{max}$, the extended Krylov method \cite{DruskinKnizhnerman98}, that alternates poles at $\infty$ with poles at $0$, and a general Krylov method with an asymptotically optimal pole sequence for Laplace-Stieltjes functions, developed in \cite{MasseiRobol20}.
The poles were selected using the interval $[\sigma_n^2, \sigma_1^2]$, with reference to the bound~\eqref{eqn:gmf-rational-bound--practical}. The function $f(z) = \sqrt{z}\log(1 + \sqrt{z})$ is such that the function
\begin{equation*}
\frac{f(\sqrt{z})}{\sqrt{z}} = \dfrac{\log(1 + \sqrt[4]{z})}{\sqrt[4]{z}}
\end{equation*}
is Laplace-Stieltjes, or equivalently, completely monotonic~\cite[Definition~1.3]{SSV-BernsteinFunctions}. This follows from~\cite[Theorem~3.7]{SSV-BernsteinFunctions} and the fact that $\log(1+z)/z$ is completely monotonic.
An approximation from above to the bound~\eqref{eqn:gmf-rational-bound--2-intervals} was evaluated using a quasi-optimal polynomial $p$ computed by replacing the uniform norm with the 2-norm on a discrete set of points in $I \cup (-I)$.
We can see in Figure~\ref{fig:rational-convergence} that the convergence of the rational Krylov method with asymptotically optimal poles closely follows the bound~\eqref{eqn:gmf-rational-bound--2-intervals}, and that the convergence rate of the Shift-and-Invert method is correctly predicted by the bound~\eqref{eqn:SI-bound-chosen-pole--SI-pole}. The convergence speed of the extended Krylov method is comparable to the one of the Shift-and-Invert method. Note that, as in the polynomial case, the bound~\eqref{eqn:SI-bound-chosen-pole--SI-pole} displayed in Figure~\ref{fig:rational-convergence} does not include the multiplicative constant.
%~ the bound \eqref{eqn:gmf-rational-bound--2-intervals} is, instead, sharp.

\begin{figure}
	\makebox[\linewidth][c]{
			\begin{subfigure}[t]{.6\textwidth}
		\begin{tikzpicture}
			\begin{semilogyaxis}[
				title = {},  % whatever name you want
				xlabel = {Iteration},
				ylabel = {Error},
				x tick label style={/pgf/number format/.cd,%
					scaled x ticks = false,
					set thousands separator={},
					fixed},
				legend pos=south west,
				ymajorgrids=true,
				grid style=dashed,legend style={font=\tiny}]
		
				\addplot[color=blue, forget plot] table {EDS_2.dat};
				\addplot[color=blue, mark=*, mark size=1pt] table {EDS_2.dat};
				\addlegendentry{{Opt. Poles}}

				\addplot[color=blue, mark size=1pt, dashed] table {ratbound_2.dat};
				\addlegendentry{{Bound (6.13)}}
								
				\addplot[color=red, forget plot] table {SI_2.dat};
				\addplot[color=red, mark=square*, mark size=1pt] table {SI_2.dat};
				\addlegendentry{{S\&I}}

				\addplot[color=red, mark size=1pt, dashed] table {SIbound_2.dat};
				\addlegendentry{{Bound (6.21)}}

				\addplot[color=green, forget plot] table {ext_2.dat};
				\addplot[color=green, mark=triangle*, mark size=1pt] table {ext_2.dat};
				\addlegendentry{{Extended}}
				
			\end{semilogyaxis}
	\end{tikzpicture}
		\caption{}
\end{subfigure}
		\begin{subfigure}[t]{.6\textwidth}
	\begin{tikzpicture}
			\begin{semilogyaxis}[
				title = {},  % whatever name you want
				xlabel = {Iteration},
				ylabel = {Error},
				legend pos=south west,
				ymajorgrids=true,
				grid style=dashed,legend style={font=\tiny}]

				\addplot[color=blue, forget plot] table {EDS_1.dat};
				\addplot[color=blue, mark=*, mark size=1pt] table {EDS_1.dat};
				\addlegendentry{{Optimal}}

				\addplot[color=blue, mark size=1pt, dashed] table {ratbound_1.dat};
				\addlegendentry{{Bound (6.13)}}
								
				\addplot[color=red, forget plot] table {SI_1.dat};
				\addplot[color=red, mark=square*, mark size=1pt] table {SI_1.dat};
				\addlegendentry{{S\&I}}

				\addplot[color=red, mark size=1pt, dashed] table {SIbound_1.dat};
				\addlegendentry{{Bound (6.21)}}

				\addplot[color=green, forget plot] table {ext_1.dat};
				\addplot[color=green, mark=triangle*, mark size=1pt] table {ext_1.dat};
				\addlegendentry{{Extended}}

			\end{semilogyaxis}
	\end{tikzpicture}
\caption{}
\end{subfigure}}
\caption{Convergence of different rational Krylov methods for the
approximation of $f^\diamond(A) \vec b$, where $A$ is a $2000 \times 2000$ matrix with logspaced singular values in the interval $[10^{-1}, 10]$ (left) or $[1, 10]$ (right), $f(z) = \sqrt{z} \log(1 + \sqrt{z})$, and $\vec b$ is a random vector.}
\label{fig:rational-convergence}
\end{figure}
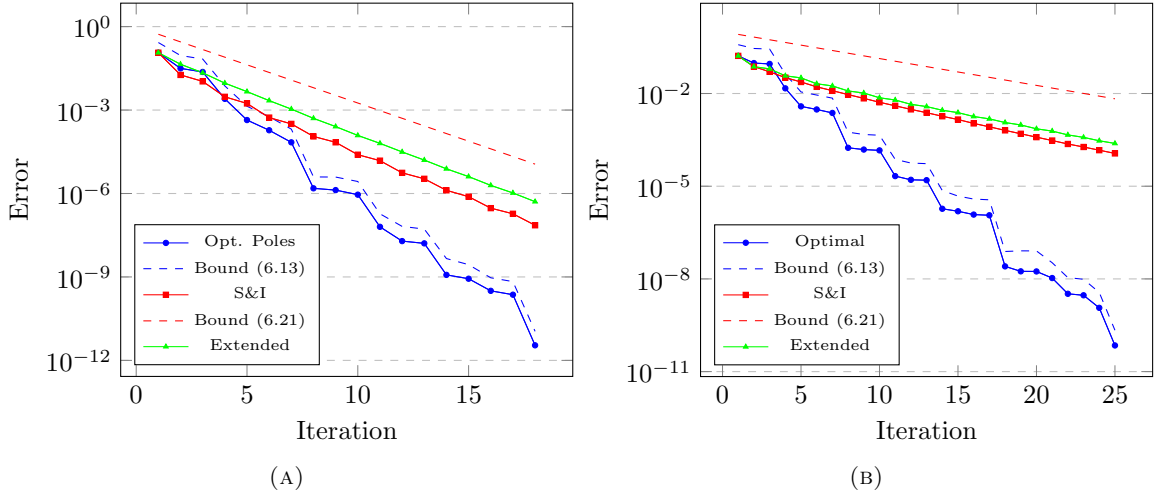

\subsection{Rectangular case}

Next, we investigate the performance of the methods in the case of a rectangular matrix $A \in \R^{m \times n}$, and the effectiveness of the strategy proposed in Remark~\ref{rem:rectang-discuss-post-poly-bound} when $m < n$ to reduce the computation of $f^\diamond(A) \vec b$ to the computation of $f^\diamond(A^T) A \vec b$. We report in Figure~\ref{fig:transp-convergence} the convergence plots of the rational Krylov method with asymtotically optimal poles, for the functions $f(z) = \sqrt{z}$ and $f(z) = z\log(z)$. We can observe that the convergence is similar for the function $z \log(z)$ (Figure~\ref{fig:transp-convergence}(b)), while there is a large benefit in using the alternative expression~\eqref{eqn:gmf-approx-transp} in the case of the function $f(z) = \sqrt{z}$. This is likely due to the fact that $\sqrt{z}$ has a large derivative close to zero, and hence roundoff errors in the smallest computed singular values of the matrix $B_k$ are extremely amplified when applying the function $f$. Indeed, we can see in Figure~\ref{fig:transp-convergence}(a) that it is not possible to get below a relative accuracy of $10^{-8}$ if we directly approximate $f^\diamond(A) \vec b$, while we can reach a relative error of about $10^{-13}$ if we use the connection with $f^\diamond(A^T) A \vec b$, since in this case the projected matrix $B_k$ has no singular values close to zero.

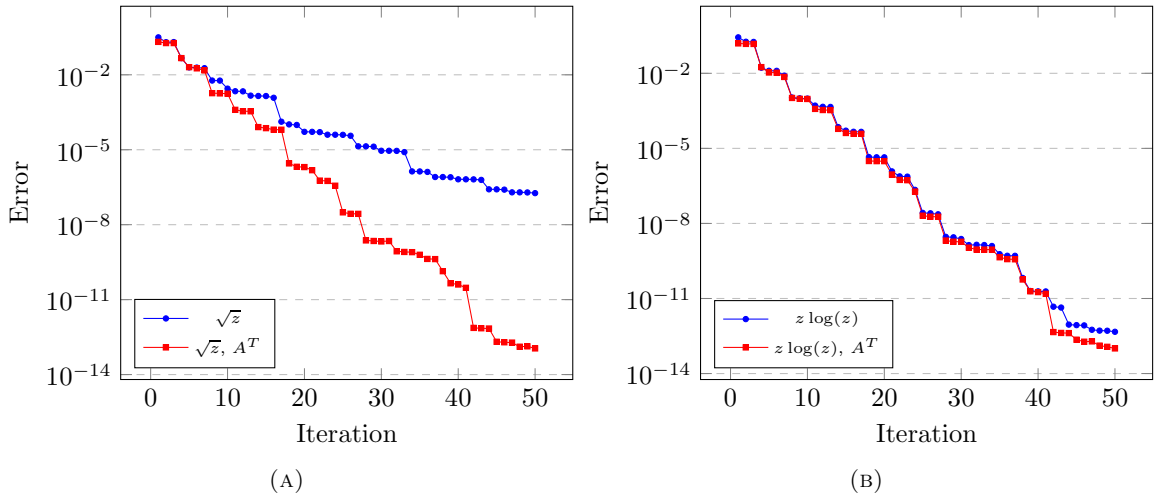
\begin{figure}
	\makebox[\linewidth][c]{
			\begin{subfigure}[t]{.6\textwidth}
		\begin{tikzpicture}
			\begin{semilogyaxis}[
				title = {},  % whatever name you want
				xlabel = {Iteration},
				ylabel = {Error},
				x tick label style={/pgf/number format/.cd,%
					scaled x ticks = false,
					set thousands separator={},
					fixed},
				legend pos=south west,
				ymajorgrids=true,
				grid style=dashed,legend style={font=\tiny}]
		
				\addplot[color=blue, mark=*, mark size=1pt] table {rect_sqrtz.dat};
				\addlegendentry{{$\sqrt{z}$}}

				\addplot[color=red, mark=square*, mark size=1pt] table {rect_sqrtz_transp.dat};
				\addlegendentry{{$\sqrt{z}$, $A^T$}}
				
			\end{semilogyaxis}
	\end{tikzpicture}
		\caption{}
\end{subfigure}
		\begin{subfigure}[t]{.6\textwidth}
	\begin{tikzpicture}
			\begin{semilogyaxis}[
				title = {},  % whatever name you want
				xlabel = {Iteration},
				ylabel = {Error},
				legend pos=south west,
				ymajorgrids=true,
				grid style=dashed,legend style={font=\tiny}]

				\addplot[color=blue, mark=*, mark size=1pt] table {rect_zlogz.dat};
				\addlegendentry{{$z \log(z)$}}

				\addplot[color=red, mark=square*, mark size=1pt] table {rect_zlogz_transp.dat};
				\addlegendentry{{$z \log(z)$, $A^T$}}

			\end{semilogyaxis}
	\end{tikzpicture}
\caption{}
\end{subfigure}}
\caption{Convergence of the rational Krylov methods with asymptotically optimal poles for
the approximation of $f^\diamond(A) \vec b$, where $A$ is a rectangular $1000 \times 1500$ matrix whose singular values are Chebyshev points of the second kind for the interval $[10^{-2}, 10]$.
The red line shows the convergence of the method described in Remark~\ref{rem:rectang-discuss-post-poly-bound}, which computes $f^\diamond(A) \vec b$ by first computing $f^\diamond(A^T) A \vec b$ and then solving a least squares problem.}
\label{fig:transp-convergence}
\end{figure}

\subsection{Finite precision issues}
In finite precision, one of the main practical problems of the Krylov methods based on a short recurrence (such as, for instance, the Lanczos method) is the loss of orthogonality in the computed basis vector. This phenomenon has been studied for the polynomial Lanczos case in \cite{paige1976error}. A brief study of the problem for the rational Lanczos case can be found in \cite{PPS21}.

As can be expected, the algorithm presented in Section~\ref{section:Rat-Golub_Kahan} also suffers from this numerical instability. However, our experiments show that this loss of orthogonality deteriorates only slightly the accuracy of the algorithm: if the poles are chosen to guarantee a moderate number of iterations for convergence, it appears that the error produced by comparing the short recurrence algorithm with the one that uses full ortogonalization remains rather small, and it stops growing after a few iterations (see Figure \ref{fig:lossOrthogonality}). This effect has been already studied in \cite{musco2018stability} for the approximation of the product between a standard matrix function and a vector by means of the polynomial Lanczos algorithm.

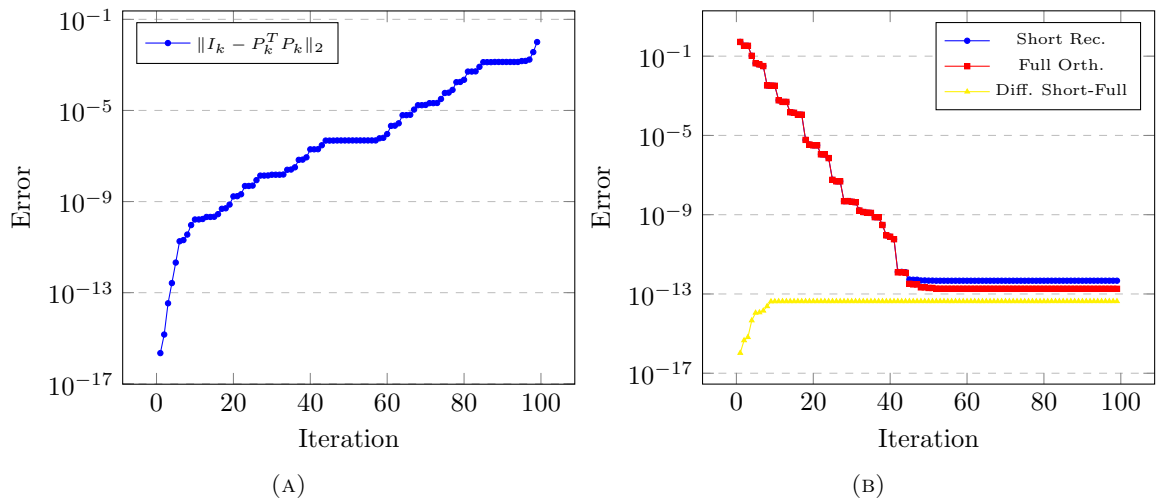
\begin{figure}
	\makebox[\linewidth][c]{
			\begin{subfigure}[t]{.6\textwidth}
		\begin{tikzpicture}
			\begin{semilogyaxis}[
				title = {},  % whatever name you want
				xlabel = {Iteration},
				ylabel = {Error},
				x tick label style={/pgf/number format/.cd,%
					scaled x ticks = false,
					set thousands separator={},
					fixed},
				legend pos=north west,
				ymajorgrids=true,
				grid style=dashed,legend style={font=\tiny}]
				\addplot[color=blue,mark=*, mark size=1pt] table {lossOrth.dat};
				\addlegendentry{{$\norm{I_k-P_k^TP_k}_2$}}
			\end{semilogyaxis}
	\end{tikzpicture}
		\caption{}
\end{subfigure}
		\begin{subfigure}[t]{.6\textwidth}
	\begin{tikzpicture}
			\begin{semilogyaxis}[
				title = {},  % whatever name you want
				xlabel = {Iteration},
				ylabel = {Error},
				legend pos=north east,
				ymajorgrids=true,
				grid style=dashed,legend style={font=\tiny}]
				\addplot[color=blue,mark=*, mark size=1pt] table {ErrShortRec.dat};
				\addlegendentry{{Short Rec.}}
				\addplot[color=red,mark=square*, mark size=1pt] table {ErrFullOrth.dat};
				\addlegendentry{{Full Orth.}}
				\addplot[color=yellow,mark=triangle*, mark size=1pt] table {diffFullShort.dat};
				\addlegendentry{{Diff. Short-Full}}
			\end{semilogyaxis}
	\end{tikzpicture}
\caption{}
\end{subfigure}}
\caption{Effects of the loss of orthogonality in the rational Golub-Kahan algorithm for the approximation of $f^\diamond(A) \vec b$, where $f(z)=\sqrt{z}$ and $A$ is a $2000 \times 2000$ matrix with logspaced singular values in the interval $[10^{-1}, 10^2]$, for the rational Krylov method with asymptotically optimal poles.
Left: loss of orthogonality when using the short recurrence. Right: comparision of the error in the approximation of $f^\diamond(A) \vec b$ when using the short recurrence or full orthogonalization of the basis vectors. In yellow we reported the norm of the difference between the two approximations.} 
\label{fig:lossOrthogonality}
\end{figure}

\section{Conclusions}
\label{sec:conclusions}
In this paper we have proposed the use of rational Krylov methods in the computation of the action of a generalized matrix function on a vector. We have developed an extension of the Golub-Kahan bidiagonalization to the rational case, that uses a short recurrence to compute the basis vectors of the rational Krylov subspace. We have proved error bounds for the computation of GMFs with polynomial and rational Krylov methods, that relate the error of approximating $f^\diamond(A)\vec b$ with the best uniform polynomial or rational approximation of the function $f$ on a real interval containing the singular values of $A$, and we have conducted experiments to investigate the sharpness of such bounds. The experiments we performed also show that rational Krylov methods are particularly effective compared to polynomial Krylov methods when the function $f$ or its derivatives have singularities close to the singular values of $A$.

\section*{Acknowledgements} The authors would like to thank Michele Benzi for his support and advice.

\FloatBarrier

\bibliographystyle{amsplain}
\bibliography{biblio_gmf}

\end{document}